\documentclass[final]{siamart190516}
\usepackage{amssymb,nicefrac,bm,upgreek,mathtools}
\usepackage[mathscr]{eucal}
\usepackage{dsfont}

\newsiamremark{remark}{Remark}
\newsiamremark{assumption}{Assumption}
\newsiamremark{notation}{Notation}
\newsiamremark{example}{Example}

\crefname{assumption}{Assumption}{Assumptions}
\newcommand{\ttup}[1]{\textup{(}#1\textup{)}}

\newcommand{\vertiii}[1]{{\left\vert\kern-0.25ex\left\vert\kern-0.25ex\left\vert #1
 \right\vert\kern-0.25ex\right\vert\kern-0.25ex\right\vert}}

\newcommand{\RR}{\mathds{R}}
\newcommand{\NN}{\mathds{N}}

\newcommand{\Rd}{{\mathds{R}^{d}}}
\DeclareMathOperator{\Exp}{\mathbb{E}}
\DeclareMathOperator{\Prob}{\mathbb{P}}
\newcommand{\D}{\mathrm{d}}
\newcommand{\E}{\mathrm{e}}
\newcommand{\Ind}{\mathds{1}} 

\newcommand{\Sob}{{\mathscr W}} 
\newcommand{\Sobl}{{\mathscr W}_{\text{loc}}} 

\newcommand{\Uadm}{{\Xi}}
\newcommand{\Act}{{\mathscr{K}}}

\newcommand{\Usm}{{\Xi_{\mathsf{sm}}}}

\newcommand{\rc}{c}

\newcommand{\sA}{{\mathscr{A}}} 
\newcommand{\sB}{{\mathscr{B}}} 
\newcommand{\Cc}{{\mathcal{C}}} 
\newcommand{\sE}{{\mathscr{E}}} 
\newcommand{\sF}{{\mathfrak{F}}} 
\newcommand{\cG}{{\mathcal{G}}} 
 
\newcommand{\cH}{{\mathcal{H}}} 
\newcommand{\eom}{{\mathcal{M}}} 
\newcommand{\Lg}{{\mathcal{L}}} 
\newcommand{\Lp}{{L}} 
\newcommand{\cP}{{\mathcal{P}}} 
\newcommand{\sR}{{L}} 
\newcommand{\cX}{{\mathcal{X}}}
\newcommand{\cZ}{{\mathcal{Z}}}

\newcommand{\varphis}{{\varphi^{}_{\mspace{-2mu}*}}}
\newcommand{\Phis}{\Phi_{\mspace{-2mu}*}}
\newcommand{\Psiv}{\Psi_{\mspace{-2mu}v}}
\newcommand{\Vst}{V_{\mspace{-2mu}*}}
\newcommand{\rhos}{{\rho^{}_*}}
\newcommand{\cPs}{{\cP^{}_{\mspace{-3mu}*}}}
\newcommand{\cPo}{{\cP^{}_{\mspace{-3mu}\circ}}}

\newcommand{\df}{\coloneqq}
\newcommand{\transp}{^{\mathsf{T}}}

\DeclareMathOperator*{\trace}{trace}

\newcommand{\grad}{\nabla}
\newcommand{\uuptau}{{\Breve\uptau}}

\newcommand{\abs}[1]{\lvert#1\rvert}
\newcommand{\norm}[1]{\lVert#1\rVert}
\newcommand{\babs}[1]{\bigl\lvert#1\bigr\rvert}

\newcommand{\bnorm}[1]{\bigl\lVert#1\bigr\rVert}

\newcommand{\compcent}[1]{\vcenter{\hbox{$#1\circ$}}}

\newcommand{\comp}{\mathbin{\mathchoice
 {\compcent\scriptstyle}{\compcent\scriptstyle}
 {\compcent\scriptscriptstyle}{\compcent\scriptscriptstyle}}}

\newcommand{\TheTitle}{A variational formula for risk-sensitive control}
\newcommand{\TheAuthors}{A. Arapostathis, A. Biswas, V.S. Borkar, and K. Suresh Kumar}
\headers{\TheTitle}{\TheAuthors}

\title{A variational characterization of the risk-sensitive
average reward for controlled diffusions on $\Rd$.}

\author{Ari Arapostathis\thanks{Department of Electrical and Computer Engineering,
The University of Texas at Austin,
EER~7.824, Austin, TX~~78712 (\email{ari@utexas.edu}).}
\and
Anup Biswas\thanks{Department of Mathematics,
Indian Institute of Science Education and Research,
Dr.\ Homi Bhabha Road, Pune 411008, India (\email{anup@iiserpune.ac.in}).}
\and
Vivek S. Borkar\thanks{Department of Electrical Engineering,
Indian Institute of Technology, Powai, Mumbai 400076, India
(\email{borkar.vs@gmail.com}).}
\and
K. Suresh Kumar\thanks{Department of Mathematics,
Indian Institute of Technology, Powai, Mumbai 400076, India
(\email{ksureshiitb@gmail.com}).}}

\begin{document}
\maketitle

\begin{abstract}
We address the variational formulation of the risk-sensitive reward problem
for non-degenerate diffusions on $\Rd$ controlled through the drift.
We establish a variational formula on the whole space and also show that
the risk-sensitive value equals the generalized principal eigenvalue
of the semilinear operator.
This can be viewed as a controlled version of the variational
formulas for principal eigenvalues of diffusion operators arising in large deviations.
We also revisit the average risk-sensitive minimization problem and by employing
a gradient estimate developed in this paper we extend earlier results to
unbounded drifts and running costs.
\end{abstract}

\begin{keywords}
principal eigenvalue,  Donsker--Varadhan functional, risk-sensitive criterion
\end{keywords}
\begin{AMS}
60J60, Secondary 60J25, 35K59, 35P15, 60F10
\end{AMS}

\section{Introduction}

In this paper we consider the risk-sensitive reward maximization problem on
$\Rd$ for diffusions controlled through the drift.
The main objective is to derive a variational formulation for the risk-sensitive
reward in the spirit of \cite{VAB}, which does so for discrete time problems on a
compact state space, and analyze the associated Hamilton--Jacobi--Bellman (HJB)
equation.
Since the seminal work of Donsker and Varadhan \cite{DoVa-72,DoVa-76b},
this problem has acquired prominence.
The variational formula derived here can be viewed as a controlled version of
the variational formulas for principal eigenvalues of diffusion operators arising
in large deviations.
For reversible diffusions, this formula can be viewed as an abstract Courant--Fischer
formula \cite{DoVa-72}. 
For general diffusions, the correct counterpart in linear algebra is the
Collatz--Wielandt formula for the principal eigenvalue of non-negative
matrices \cite[Chapter~8]{Meyer}.
For its connection with the large deviations theory for finite Markov chains and
an equivalent variational description, see \cite{Dembo}. 

There has been considerable interest to generalize this theory to a natural class of
nonlinear self-maps on positive cones of finite or infinite dimensional spaces.
The first task is to establish the existence and where possible, uniqueness of the
principal eigenvalue and eigenvector (the latter modulo a scalar multiple as usual),
that is, a nonlinear variant of the Perron--Frobenius theorem in the finite dimensional
case and its generalization, the Krein--Rutman theorem, in Banach spaces.
This theory is carried out in, e.g., \cite{Lemmens,Ogiwara}.
The next problem is to derive an abstract Collatz--Wielandt formula for the principal
eigenvalue \cite{Gaubert}.
In bounded domains, a Collatz--Wielandt formula for the Dirichlet principal eigenvalue
of a convex nonlinear operator is obtained in \cite{Armstrong-09}.
Our first objective coincides with this, albeit for 
Feynman--Kac operators arising in risk-sensitive control that we introduce later.
For risk-sensitive \emph{reward} processes, that is, the problem of maximizing the
asymptotic growth rate for the risk-sensitive reward in discrete time problems,
one can go a step further and give an explicit characterization of the principal
eigenvalue as the solution of a concave maximization problem \cite{VAB}.
The objective of this article is to carry out this program for controlled diffusions.

At this juncture, it is worthwhile to underscore the difference between reward
maximization and cost minimization problems with risk-sensitive criteria.
Unlike the more classical criteria such as ergodic or discounted, they cannot be
converted from one to the other by a sign flip. The cost minimization criterion,
after a logarithmic transformation applied to its HJB equation, leads to the
Isaacs equation for a zero-sum stochastic differential game \cite{FM}.
An identical procedure applied to the reward maximization problem would lead
to a \emph{team} problem wherein the two agents seek to maximize the same payoff
\emph{non-cooperatively}.
The latter in particular implies that their decisions at any time are conditionally
independent given the state (more generally, the past history).
Our approach leads to a concave maximization problem, an immense improvement
with potential implications for possible numerical schemes.
This does not seem possible for the cost minimization problem.
Thus the complexity of the latter is much higher.
Recently, a risk-sensitive maximization problem
is also studied in \cite{BS-20} under a blanket geometric stability condition.
In the present paper we do not impose any blanket stability on the controlled processes.

We first establish these results for reflected diffusions in a bounded domain,
for which the nonlinear Krein--Rutman theorem of \cite{Ogiwara} paves the way.
This is not so if the state space is all of $\Rd$.
Extension to the whole space turns out to be quite involved due to the lack
of compactness.
Even the well-posedness of the underlying nonlinear eigenvalue problem is pretty tricky.
Hence we proceed via the infinite volume limit of the finite volume problems.
This leads to an abstract Collatz--Wielandt formula and an
abstract Donsker--Varadhan formula.
More specifically, in \cref{T3.1} we show that the generalized eigenvalue of
the semilinear operator is simple, and identify some useful properties of its
eigenvector.
We proceed to prove equality between the risk-sensitive
value and the generalized principal eigenvalue in \cref{T3.2},
which also establishes a verification of optimality criterion.
The general result for the variational formula is in \cref{P4.1}, followed by
more specialized results in  \cref{T4.1,T4.2}.
In the process of deriving these results,
we present some techniques that may have wider applicability.
Most prominent of these is perhaps the gradient estimate in \cref{L4.1} for operators
with measurable coefficients.

Lastly, in \cref{S5} we revisit the risk-sensitive minimization problem,
and with the aid of \cref{L4.1} we improve the main result
in \cite{AB-18} by extending it to unbounded drifts and running costs,
under suitable growth conditions (see \cref{A5.1}).

\subsection{A brief summary of the main results}\label{S1.1}
We summarize here the results concerning the variational formula
on the whole space.
We consider a controlled diffusion in $\Rd$ of the form
\begin{equation*}
\D X_t \,=\, b(X_t,\xi_t)\,\D t + \upsigma (X_t)\,\D W_t
\end{equation*}
defined in a complete
probability space $(\Omega,\sF,\Prob)$.
The process $W$ is a $d$-dimensional standard Wiener process independent
of the initial condition $X_{0}$, and
the control process $\{\xi_t\}_{t\ge0}$ lives in a compact metrizable space $\Act$.
We impose a standard set of assumptions on the coefficients 
which guarantee existence and uniqueness of strong solutions under
all admissible controls.
Namely, local Lipschitz continuity in $x$
and at most affine growth of $b$ and $\upsigma$,
and local non-degeneracy of $a\df\upsigma\upsigma\transp$ (see \cref{A3.1}\,(i)).
But we do not impose any ergodicity assumptions on the controlled diffusion.
The process $\{X_t\}_{t\ge0}$ could be transient.

We let $\rc\colon\Rd\times\Act\to\RR$ be a continuous running reward function,
which is assumed bounded from above,
and define the \emph{optimal risk-sensitive value} $J_*$ by
\begin{equation*}
J_* \,\df\, \sup_{\{\xi_t\}_{t\ge0}}\;\liminf_{T\to\infty}\, \frac{1}{T}\,
\log \Exp \Bigl[\E^{\int^T_0 \rc(X_t,\xi_t)\,\D t} \Bigr]\,,
\end{equation*}
where the supremum is over all admissible controls, and $\Exp$ denotes
the expectation operator.
This problem is translated to an ergodic control problem for the operator
$\sA\colon \Cc^{2}(\Rd) \to \Cc(\Rd\times\Act\times\Rd)$, defined by
\begin{equation}\label{EsA}
\sA\phi(x,\xi,y)\,\df\,\frac{1}{2}\trace\left(a(x)\nabla^{2}\phi(x)\right)
+ \bigl\langle b(x,\xi)+ a(x)y, \nabla \phi(x)\bigr\rangle\,,
\end{equation}
where $\nabla^{2}$ denotes the Hessian,
and $a(x)=\upsigma(x)\upsigma\transp(x)$,
that seeks to maximize the average value of the functional
\begin{equation}\label{EsR}
\sR(x,\xi,y) \,\df\, \rc(x,\xi) - \frac{1}{2}\abs{\upsigma\transp(x) y}^2\,,\quad
(x,\xi,y)\in \Rd\times\Act\times\Rd\,.
\end{equation}
We first show that the generalized principal eigenvalue $\lambda_*$
(see \cref{E-princ2})
of the maximal operator
\begin{equation}\label{EcG}
\cG f(x) \,\df\, \frac{1}{2}\trace\left(a(x)\nabla^{2}f(x)\right)
+ \max_{\xi\in\Act}\, \bigl[\bigl\langle b(x,\xi),
\grad f(x)\bigr\rangle + \rc(x,\xi) f(x)\bigr]
\end{equation}
is simple.
An important hypothesis for this is that $\rc-\lambda_*$ is negative and
bounded from above away from zero on the complement of some compact set
(see \cref{A3.1}\,(iii)).
This is always satisfied if $-\rc$ is an inf-compact function
(i.e., the sublevel sets $\{-\rc \le \kappa\}$ are compact, or empty, in $\Rd\times\Act$
for each $\kappa\in\RR$), or if $\rc$ is a positive function vanishing at infinity
and the process $\{X_t\}_{t\ge0}$ is recurrent under some
stationary Markov control.
Let the positive function
$\Phis\in\Cc^2(\Rd)$, normalized as $\Phis(0)=1$ to render it unique,
denote the principal eigenvector, that is, $\cG\Phis=\lambda_*\Phis$,
and define $\varphis=\log\Phis$.
The function
\begin{equation}\label{Eentropy}
\cH(x)\,\df\,\frac{1}{2}\,\babs{\upsigma\transp(x)\grad \varphis(x)}^2\,,
\quad x\in\Rd\,,
\end{equation}
plays a very important role in the analysis, and can be interpreted
as an \emph{infinitesimal relative entropy rate} (see \cref{S4}).
To keep the notation simple, we define $\cZ \df \Rd\times\Act\times\Rd$, and
use the single variable $z=(x,\xi,y)\in\cZ$.
Let $\cP(\cZ)$ denote the set of probability measures
on the Borel $\sigma$-algebra of $\cZ$, and 
$\eom_A$ denote the set of \emph{infinitesimal ergodic occupation measures}
for the operator $\sA$ defined by
\begin{equation}\label{Eeom}
\eom_{\sA}\,\df\,\biggl\{ \mu\in \cP(\cZ)\,\colon
\int_{\cZ} \sA f(z)\,\mu(\D{z}) \,=\, 0\quad \forall\, f\in\Cc^{2}_c(\Rd)\biggr\}\,,
\end{equation}
where $\Cc^2_c(\Rd)$ is the class of functions
in $\Cc^2(\Rd)$ which have compact support.
We also define
\begin{equation}\label{Emufinite}
\begin{aligned}
\cPs(\cZ) &\,\df\,
\biggl\{\mu\in \cP(\cZ)\,\colon
\int_{\cZ} \cH(x)\,\mu(\D{x},\D{\xi},\D{y}) <\infty\biggr\}\,,\\
\cPo(\cZ) &\,\df\,
\biggl\{ \mu\in \cP(\cZ)\,\colon
\int_{\cZ} \sR(z)\,\mu(\D{z}) > -\infty\biggr\}\,.
\end{aligned}
\end{equation}
Then, under the mild hypotheses of \cref{A3.1}, we show in \cref{P4.1} that
\begin{equation}\label{E-var}
\begin{aligned}
J_* \,=\, \lambda_* &\,=\, \adjustlimits\sup_{\mu\in\cPs(\cZ)}
\inf_{g \in \Cc^2_c(\Rd)}\,\int_{\cZ} \bigl(\sA g(z)+\sR(z)\bigr)\,\mu(\D{z})\\
&
\,=\, \max_{\mu\in\eom_{\sA}\cap\cPs(\cZ)}\,\int_{\cZ} \sR(z)\,\mu(\D{z})\,.
\end{aligned}
\end{equation}

We next specialize the results to the case where the diffusion matrix
$a$ is bounded and uniformly elliptic (see \cref{A4.1}), and
show in \cref{T4.1} that under any of the hypotheses of
\cref{A4.2} we have $\eom_{\sA}\cap\cPo(\cZ)\subset
\cPs(\cZ)$.
This permits us to replace $\cPs(\cZ)$ with
$\cP(\cZ)$ and $\eom_\sA\cap\cPs(\cZ)$ with
$\eom_\sA$ in the second and third equalities of \cref{E-var}, respectively.
We note here that if $a$ is bounded and uniformly elliptic, then
\cref{A4.2} is satisfied when either
$-\rc$ is inf-compact, or $\langle b,x\rangle^-$ has subquadratic growth,
or $\frac{\abs{b}^2}{1+\abs{\rc}}$ is bounded.

We also show that if $\frac{\cH}{1+\abs{\varphis}}$ is bounded
(see \cref{L4.3} for explicit conditions on the parameters under
which this holds), then we can commute
the `$\sup$' and the `$\inf$' to obtain
\begin{equation*}
J_* \,=\, \adjustlimits\inf_{g \in \Cc^2_c(\Rd)} \sup_{\mu\in\cP(\cZ)}\,
\int_{\cZ} \bigl(\sA g(z)+\sR(z)\bigr)\,\mu(\D{z}) \,.
\end{equation*} 
Also, in \cref{T4.2}, we establish the variational formula
over the class of functions in $\Cc^2(\Rd)$ whose partial derivatives
up to second order have
at most polynomial growth in $\abs{x}$.

\subsection{Notation}
The standard Euclidean norm in $\RR^{d}$ is denoted by $\abs{\,\cdot\,}$,
and $\NN$ stands for the set of natural numbers.
The closure, the boundary and the complement
of a set $A\subset\RR^{d}$ are denoted
by $\Bar{A}$, $\partial{A}$ and $A^{c}$, respectively.
We denote by $\uptau(A)$ the \emph{first exit time} of the process
$\{X_{t}\}$ from the set $A\subset\RR^{d}$, defined by
\begin{equation*}
\uptau(A) \,\df\, \inf\,\{t>0\,\colon\, X_{t}\not\in A\}\,.
\end{equation*}
The open ball of radius $r$ in $\RR^{d}$, centered at $x\in\Rd$,
is denoted by $B_{r}(x)$, and $B_r$ is the ball centered at $0$.
We let $\uptau_{r}\df \uptau(B_{r})$,
and $\uuptau_{r}\df \uptau(B^{c}_{r})$.
For a Borel space $Y$, $\cP(Y)$ denotes the set of probability measures
on its Borel $\sigma$-algebra.

The term \emph{domain} in $\RR^{d}$
refers to a nonempty, connected open subset of the Euclidean space $\RR^{d}$.
For a domain $D\subset\RR^{d}$,
the space $\Cc^{k}(D)$ ($\Cc^{k}_b(D)$)
refers to the class of all real-valued functions on $D$ whose partial
derivatives up to order $k$ exist and are continuous (and bounded).
In addition $\Cc_c^k(D)$ denotes the class of functions in $\Cc^k(D)$ that
have compact support.
The space $\Lp^{p}(D)$, $p\in[1,\infty)$, stands for the Banach space
of (equivalence classes of) measurable functions $f$ satisfying
$\int_{D} \abs{f(x)}^{p}\,\D{x}<\infty$, and $\Lp^{\infty}(D)$ is the
Banach space of functions that are essentially bounded in $D$.
The standard Sobolev space of functions on $D$ whose generalized
derivatives up to order $k$ are in $\Lp^{p}(D)$, equipped with its natural
norm, is denoted by $\Sob^{k,p}(D)$, $k\ge0$, $p\ge1$.

In general, if $\mathcal{X}$ is a space of real-valued functions on $Q$,
$\mathcal{X}_{\mathrm{loc}}$ consists of all functions $f$ such that
$f\varphi\in\mathcal{X}$ for every $\varphi\in\Cc_{c}^{\infty}(Q)$,
the space of smooth functions on $Q$ with compact support.
In this manner we obtain for example the space $\Sobl^{2,p}(Q)$.

We adopt the notation
$\partial_{t}\df\tfrac{\partial}{\partial{t}}$, and for $i,j\in\NN$,
$\partial_{i}\df\tfrac{\partial~}{\partial{x}_{i}}$ and
$\partial_{ij}\df\tfrac{\partial^{2}~}{\partial{x}_{i}\partial{x}_{j}}$,
and use the standard summation rule that
repeated subscripts and superscripts are summed from $1$ through $d$.

\section{The problem on a bounded domain}

In this section, we consider the risk-sensitive reward maximization with state dynamics
given by a reflected diffusion on a bounded $\Cc^2$ domain $Q\subset\Rd$
with co-normal direction of reflection.
In particular, the dynamics are given by
\begin{equation}\label{E-sde}
\D X_t \,=\, b(X_t,\xi_t)\,\D t + \upsigma (X_t)\,\D W_t
- \gamma (X_t)\,\D \eta_t\,,
\end{equation}
where $\eta_t$ denotes the local time of the process $X$ on the boundary
$\partial Q$. 
The random processes in \cref{E-sde} live in a complete
probability space $(\Omega,\sF,\Prob)$.
The process $W=(W_t)_{t\ge0}$ is a $d$-dimensional standard Wiener process independent
of the initial condition $X_{0}$.
The control process $\xi=(\xi_t)_{t\ge0}$ takes values in a compact,
metrizable set $\Act$, and
$\xi_t(\omega)$ is jointly measurable in
$(t,\omega)\in[0,\infty)\times\Omega$.
The set of \emph{admissible controls} $\Uadm$ consists of the
control processes $\xi$ that are \emph{non-anticipative}:
for $s < t$, $W_{t} - W_{s}$ is independent of
\begin{equation}\label{E-sF}
\sF_{s} \,\df\,\text{the completion of~} \sigma\{X_{0},\xi_r,W_r,\,r\le s\}
\text{~relative to~}(\sF,\Prob)\,.
\end{equation}
Concerning the coefficients of the equation, we assume the following:
\begin{enumerate}
\item[(i)]
The drift $b$ is a continuous map from $\overline{Q}\times\Act$ to $\Rd$,
and Lipschitz in its first
argument uniformly with respect to the second.

\item[(ii)]
The \emph{diffusion matrix} $\upsigma\colon \overline{Q} \to\RR^{d\times d}$
is continuously
differentiable, its derivatives are H\"older continuous, and is non-degenerate in
the sense that the minimum eigenvalue of
$a(x)=\bigl[a^{ij}(x)\bigr]\df\upsigma(x)\upsigma\transp(x)$
on $Q$ is bounded away from zero.

\item[(iii)]
The \emph{reflection direction}
$\gamma = [\gamma_{1}(x), \dotsc, \gamma_d(x)]\transp \colon\Rd \to \Rd$ is co-normal,
that is,
$\gamma$ is given by
\begin{equation*}
\gamma_i(x) \,=\, \sum_{j=1}^d a^{ij}(x) n_{j}(x)\,,\quad x\in\partial Q\,,
\end{equation*}
where $\vec n(x)= [n_{1}(x), \dotsc, n_d(x)]\transp$ is the unit outward normal.
\end{enumerate}

We let $\Usm$ denote the set of stationary Markov controls, that is,
the set of Borel measurable functions $v\colon\Rd\to\Act$.
Given $\xi\in\Uadm$, the stochastic differential equation in \cref{E-sde}
has a unique strong solution.
The same is true for the class of Markov controls \cite[Chapter~2]{ABG}.
Let $\Prob^x_\xi$ and $\Exp^x_\xi$ denote the
probability measure and expectation operator on the canonical space of the
process controlled under $\xi\in\Uadm$, with initial condition $X_0=x$.

Given a continuous \emph{reward function} $\rc\colon\overline{Q}\times\Act\to\RR$,
which is Lipschitz continuous in its first
argument uniformly with respect to the second,
the objective of the risk-sensitive reward problem is to maximize
\begin{equation}\label{E-JQ}
J^x_\xi(\rc;Q) \,=\, \liminf_{T\to\infty}\, \frac{1}{T}\,
\log \Exp^x_\xi \Bigl[\E^{\int^T_0 \rc(X_t,\xi_t)\,\D t} \Bigr]\,,\quad x\in Q\,,
\end{equation}
over all admissible controls $\xi\in\Uadm$.
We define
\begin{equation}\label{E-J*}
J^x_*(\rc;Q) \,\df\, \sup_{\xi\in\Uadm}\,J^x_\xi(\rc;Q)\,,\quad x\in Q\,,\quad
\text{and\ \ } J_*(\rc;Q) \,\df\, \sup_{x\in Q}\,J^x_*(\rc;Q)\,.
\end{equation}
The solution of this problem shows that $J^x_*(\rc;Q)$ does not depend on $x$.

We let
\begin{equation*}
\Cc^{2}_{\gamma}(\overline{Q}) \,\df\, \bigl\{ f \in \Cc^{2}(\overline{Q})\,\colon\,
\langle\nabla f, \gamma\rangle \,=\, 0 \text{\ on\ } \partial{Q} \bigr\}\,,
\end{equation*}
and $\Cc^{2}_{\gamma,+}(\overline{Q})$ denote its subspace consisting of nonnegative
functions.

For $f \in \Cc^{2}(\overline{Q})$, and $\xi\in\Act$, we define
\begin{equation}\label{E-gen}
\begin{aligned}
\Lg_\xi f(x) &\,\df\, \tfrac{1}{2}\trace\left(a(x)\nabla^{2}f(x)\right)
+ \bigl\langle b(x,\xi),\grad f(x)\bigr\rangle\,,\\
\cG f(x) &\,\df\, \tfrac{1}{2}\trace\left(a(x)\nabla^{2}f(x)\right)
+ \max_{\xi\in\Act}\, \bigl[\bigl\langle b(x,\xi),
\grad f(x)\bigr\rangle + \rc(x,\xi) f(x)\bigr]\,.
\end{aligned}
\end{equation}

We summarize some results from \cite{ABK-16} that are needed in \cref{T2.1} below.
Without loss of generality we assume that $0\in Q$.

Consider the operator $S_t\colon\Cc(\overline{Q})\to\Cc(\overline{Q})$, $t\in\RR_+$,
defined by
\begin{equation*}
S_{t}f(x)\,\df\,\sup_{\xi\in\Uadm}\,
\Exp^x_\xi\Bigl[e^{\int_{0}^t \rc(X_s,\xi_s)\,\D{s}}f(X_t)\Bigr]\,.
\end{equation*}
The characterization of $S_t$ is exactly analogous to \cite[Theorem~3.2]{ABK-16},
which considers the minimization problem (see also \cite[Remark~4.2]{ABK-16}).
Specifically, for each $f \in C^{2+\delta}_{\gamma}(\overline{Q})$,
and $T>0$, the quasi-linear parabolic p.d.e.
$\partial_t\,u(t,x) = \cG u(t,x)$
in $(0,T]\times Q$,
with $u(0,x) = f(x)$ for all $x \in \overline{Q}$, and
$\langle\nabla u(t,x), \gamma(x)\rangle = 0$ for all
$(t,x) \in (0,T]\times\partial{Q}$,
has a unique solution in
$\Cc^{1+\nicefrac{\delta}{2},2+\delta}\bigl([0,T]\times\overline{Q}\bigr)$.
This solution has the stochastic representation
$u(t,x) \,=\, S_{t}f(x)$ for all $(t,x)\in[0,T]\times\overline{Q}$.

Following the analysis in \cite{ABK-16} we obtain the following characterization
of $J_*(\rc;Q)$ defined in \cref{E-J*}.

\begin{theorem}\label{T2.1}
There exists a unique pair
$(\rho,V) \in \RR\times \Cc^{2}_{\gamma,+}(\overline{Q})$
which solves
\begin{equation}\label{ET2.1A}
\cG V \,=\,\rho V \text{\ \ in\ }Q\,,\qquad
\langle\nabla V,\gamma\rangle \,=\, 0 \text{\ \ on\ } \partial{Q}\,,
\quad\text{and\ \ } V(0)\,=\,1\,.
\end{equation}
Also,
$S_{t}V(x)=e^{\rho t}V(x)$, for $(x,t)\in\overline{Q}\times[0,\infty)$.
In addition, we have
\begin{equation*}
J^x_*(\rc;Q)\,=\,J_*(\rc;Q)
\,=\,\rho\qquad \forall\, x\in Q\,,
\end{equation*}
and
\begin{equation}\label{ET2.1B}
\rho \,=\, \adjustlimits\inf_{f \in \Cc^{2}_{\gamma,+}(\overline{Q}),\, f>0\;}
\sup_{x\in\overline{Q}}\;\frac{\cG f(x)}{f(x)} 
= \adjustlimits\sup_{f \in \Cc^{2}_{\gamma,+}(\overline{Q}),\, f>0\;}
\inf_{x\in\overline{Q}}\;\frac{\cG f(x)}{f(x)}\,.
\end{equation}
\end{theorem}

\begin{proof}
\Cref{ET2.1B} is the result in \cite[Lemma~2.1]{ABK-16}, while the other assertions
follow from Lemma~4.5 and Remark~4.2 in \cite{ABK-16}.
\end{proof}

\subsection{A variational formula}

Define
\begin{equation*}
\sR(x,\xi,y) \,\df\, \rc(x,\xi) - \frac{1}{2}\abs{\upsigma\transp(x) y}^2\,,\quad
(x,\xi,y)\in \overline{Q}\times\Act\times\Rd\,,
\end{equation*}
and an operator
$\sA\colon \Cc^{2}_{\gamma}(\overline{Q}) \to \Cc(\Rd\times\Act\times\Rd)$ by
\begin{equation*}
\sA\phi(x,\xi,y)\,\df\,\frac{1}{2}\trace\left(a(x)\nabla^{2}\phi(x)\right)
+ \bigl\langle b(x,\xi)+ a(x)y, \nabla \phi(x)\bigr\rangle\,.
\end{equation*}
It is important to note that if $f \in \Cc^{2}_{\gamma,+}(\overline{Q})$ is a positive
function and $g=\log f$, then
\begin{equation*}
\frac{\cG f(x)}{f(x)} \,=\, \adjustlimits\max_{\xi\in\Act}\max_{y\in\Rd}\;
\bigl[\sA g(x,\xi,y) + \sR(x,\xi,y)\bigr]\,.
\end{equation*}
Thus, we obtain from \cref{ET2.1B} that
\begin{align}
\rho &\,=\, \adjustlimits\inf_{g \in \Cc^{2}_{\gamma}(\overline{Q})\,}
\sup_{x\in\overline{Q}\;}\sup_{\xi\in\Act,\,y\in\Rd}\,
\Bigl(\sA g(x,\xi,y) + \sR(x,\xi,y) \Bigr)
\label{E-infsup}\\
&\,=\, \adjustlimits\sup_{g \in \Cc^{2}_{\gamma}(\overline{Q})\,}
\inf_{x\in\overline{Q}\;}\sup_{\xi\in\Act,\,y\in\Rd}\,
\Bigl(\sA g(x,\xi,y) + \sR(x,\xi,y)\Bigr)
\nonumber\,.
\end{align}
We let
\begin{equation}\label{E-F}
F(g,\mu) \,\df\,
\int_{\overline{Q}\times\Act\times\Rd}
\bigl(\sA g(x,\xi,y)+\sR(x,\xi,y)\bigr)\,\mu(\D{x},\D{\xi},\D{y})
\end{equation}
for $g\in\Cc^{2}_{\gamma}(\overline{Q})$
and $\mu\in\cP(\overline{Q}\times\Act\times\Rd)$.

It is clear that \cref{E-infsup} can be written as
\begin{equation}\label{E-infsup1}
\rho \,=\, \adjustlimits\inf_{g\in\Cc^{2}_{\gamma}(\overline{Q})}
\sup_{\mu\in\cP(\overline{Q}\times\Act\times\Rd)}\,F(g,\mu)\,.
\end{equation}

Let $\eom_{\sA,Q}$ denote the class of infinitesimal
ergodic occupation measures for the operator $\sA$, defined by
\begin{equation*}
\eom_{\sA,Q}\,\df\,\biggl\{
\mu\in \cP(\overline{Q}\times\Act\times\Rd)\,\colon
\int_{\overline{Q}\times\Act\times\Rd} \sA f\,\D\mu=0\quad
\forall\, f\in\Cc^{2}_{\gamma}(\overline{Q})\biggr\}\,.
\end{equation*}
Implicit in this definition is the requirement that
$\int \abs{\sA f}\,\D\mu<\infty$ for all $f\in\Cc^{2}_{\gamma}(\overline{Q})$
and $\mu\in\eom_{\sA,Q}$.
We have the following result.
\begin{theorem}\label{T2.2}
It holds that
\begin{equation}\label{ET2.2A}
\rho \,=\, \adjustlimits\inf_{g \in \Cc^{2}_{\gamma}(\overline{Q})}
\sup_{\mu\in\cP(\overline{Q}\times\Act\times\Rd)}\,F(g,\mu)
\,=\,
\adjustlimits\sup_{\mu\in\cP(\overline{Q}\times\Act\times\Rd)}
\inf_{g \in \Cc^{2}_{\gamma}(\overline{Q})}\,F(g,\mu)\,.
\end{equation}
Moreover, $\cP(\overline{Q}\times\Act\times\Rd)$ may be replaced with
$\eom_{\sA,Q}$ in \cref{ET2.2A}, and thus
\begin{equation*}
\rho \,=\, \sup_{\mu\in\eom_{\sA,Q}}\,
\int_{\overline{Q}\times\Act\times\Rd} \sR(x,\xi,y)\,\mu(\D{x},\D{\xi},\D{y})\,.
\end{equation*}
\end{theorem}

\begin{proof}
The first equality in \cref{ET2.2A} follows by \cref{E-infsup1}.
We continue to prove the rest of the assertions.
First note that
\begin{equation*}
\adjustlimits\sup_{\mu\in\cP(\overline{Q}\times\Act\times\Rd)}
\inf_{g \in \Cc^{2}_{\gamma}(\overline{Q})}\,F(g,\mu)
\,=\, \Hat\rho
\,\df\, \sup_{\mu\in\eom_{\sA,Q}}\,
\int_{\overline{Q}\times\Act\times\Rd} \sR(x,\xi,y)\,\mu(\D{x},\D{\xi},\D{y})\,,
\end{equation*}
because the infimum on the left hand side is $-\infty$ for $\mu \notin \eom_{\sA, Q}$.
It follows by \cref{E-infsup1} that $\Hat\rho\le\rho$.
Let $v_*$ be a measurable selector from the maximizer of \cref{ET2.1A}, that is,
\begin{equation*}
\bigl\langle b\bigl(x,v_*(x)\bigr),\grad V(x)\bigr\rangle
+ \rc\bigl(x,v_*(x)\bigr) V(x)\,=\,
\max_{\xi\in\Act}\,
\bigl[\bigl\langle b(x,\xi),\grad V(x)\bigr\rangle + \rc(x,\xi) V(x)\bigr]\,.
\end{equation*}
With $\phi\df\log V$, \cref{ET2.1A} takes the form
\begin{equation}\label{PT2.2Xa}
\sA\phi\bigl(x,v_*(x),\grad \phi(x)\bigr) + \sR\bigl(x,v_*(x),\nabla\phi(x)\bigr)
\,=\,\rho\,.
\end{equation}
The reflected diffusion
with drift $b\bigl(x,v_*(x)\bigr)+a(x)\grad\phi(x)$ is of course
exponentially ergodic.
Let $\eta_*$ denote its invariant probability measure.
Then, \cref{PT2.2Xa} implies
that
\begin{equation}\label{PT2.2Xb}
\int_Q \sR\bigl(x,v_*(x),\nabla\phi(x)\bigr)\,\eta_*(\D{x})\,=\,\rho\,.
\end{equation}
Let $\mu_*\in\cP(\overline{Q}\times\Act\times\Rd)$ be defined by
\begin{equation*}
\mu_*(\D{x},\D{\xi},\D{y})\,\df\, \eta_*(\D{x})\,\delta_{v_*(x)}(\D{\xi})\,
\delta_{\nabla\phi(x)}(\D{y})\,,
\end{equation*}
where $\delta_y$ denotes the Dirac mass at $y$.
Then $\mu_*$ is an ergodic occupation measure for the controlled reflected diffusion
with drift $b(x,\xi)+a(x)y$, and thus $\mu_*\in\eom_{\sA,Q}$.
Let $g \in \Cc^{2}_{\gamma}(\overline{Q})$ be arbitrary.
Then
\begin{equation*}
F(g,\mu_*)\,=\,
\int_{\overline{Q}\times\Act\times\Rd} \sR(x,\xi,y)\,\mu_*(\D{x},\D{\xi},\D{y})
\,=\,\rho\,,
\end{equation*}
where the second equality follows by \cref{PT2.2Xb}.
Thus $\Hat\rho\ge\rho$, and since we have already asserted the reverse inequality,
we must have equality.
This establishes \cref{ET2.2A}, and also proves the last assertion
of the theorem.
\end{proof}

\section{The risk-sensitive reward problem on \texorpdfstring{$\Rd$}{}}

In this section we study the risk-sensitive reward maximization problem
on $\Rd$.
We consider a controlled diffusion of the form
\begin{equation}\label{E-sde1}
\D X_t \,=\, b(X_t,\xi_t)\,\D t + \upsigma (X_t)\,\D W_t\,.
\end{equation}
All random processes in \cref{E-sde1} live in a complete
probability space $(\Omega,\sF,\Prob)$.
The control process $\{\xi_t\}_{t\ge0}$ lives in a compact metrizable space $\Act$.

We approach the problem in $\Rd$ as a limit of Dirichlet or Neumann eigenvalue
problems on balls $B_r$, $r>0$.
Differentiability of the matrix $a$ can be relaxed here.
Consider the eigenvalue problem on a ball $B_r$, with Neumann boundary
conditions, and the reflection direction along the exterior normal
$\vec n(x)$ to $B_r$ at $x$. The drift $b:\Bar{B}_r\times\Act\to\Rd$ is continuous,
and Lipschitz in its first argument uniformly with respect to the second.
The diffusion matrix $a$ is Lipschitz continuous on $\Bar{B}_r$ and
non-degenerate.
Let $\rho_r$ denote the principal eigenvalue on $B_r$ under
Neumann boundary conditions of the operator
$\cG$ defined in \cref{E-gen}.
We refer to $\rho_r$ as the \emph{Neumann eigenvalue} on $B_r$.
It follows from the results in \cite{Patrizi-09}
(see in particular Theorems~5.1, 6.6, and Proposition~7.1) that there exists
a unique $V_r\in\Cc^2(B_r)\cap \Cc^{0,1}(\Bar{B}_r)$, with $V_r>0$ on $B_r$
and $V_r(0)=1$, solving
\begin{equation}\label{E-HJBN}
\tfrac{1}{2}\trace\left(a(x)\nabla^{2}V_r(x)\right)
+ \max_{\xi\in\Act}\, \bigl[\bigl\langle b(x,\xi),
\grad V_r(x)\bigr\rangle + \rc(x,\xi) V_r(x)\bigr]\,=\, \rho_r V_r(x)\,,
\end{equation}
and $\langle \nabla V_r(x),\vec n(x)\rangle=0$ on $\partial B_r$.
We also refer the reader to \cite[Theorem~12.1, p.~195]{Lady}.

We adopt the following structural hypotheses on the
coefficients of \cref{E-sde1} and the
reward function $\rc$ have the following structural properties. 

\begin{assumption}\label{A3.1}
\begin{enumerate}
\item[\ttup i]
The drift $b\colon\RR^{d}\times\Act\to\RR^{d}$ is continuous,
 and for some constant $C_R>0$ depending on $R>0$, we have
\begin{align}\label{E-growth}
\abs{b(x,\xi) - b(y,\xi)} + \norm{\upsigma(x) - \upsigma(y)} &\,\le\,C_{R}\,\abs{x-y}
\qquad\forall\,x,y\in B_R\,,\ \forall\, \xi\in\Act\,,\nonumber\\
\sum_{i,j=1}^{d} a^{ij}(x)\zeta_{i}\zeta_{j}
&\,\ge\, C^{-1}_{R} \abs{\zeta}^{2}
\qquad\forall\, (x,\zeta)\in B_{R}\times\Rd\,,\nonumber
\intertext{and}
\abs{b(x,\xi)}^2 + \norm{\upsigma(x)}^{2} &\,\le\, C_0
\bigl(1 + \abs{x}^{2}\bigr) \qquad \forall\, (x,\xi)\in\RR^{d}\times\Act\,,
\end{align}
where $\norm{\upsigma}\df\bigl(\trace\, \upsigma\upsigma\transp\bigr)^{\nicefrac{1}{2}}$
denotes the Hilbert--Schmidt norm of $\upsigma$.

\item[\ttup{ii}]
The reward function $\rc\colon\Rd\times\Act\to\RR$ is continuous
and locally Lipschitz in its first argument
uniformly with respect to $\xi\in\Act$, is bounded from above in $\Rd$,
and $x\mapsto\max_{\xi\in\Uadm}\,\abs{\rc(x,\xi)}$ has polynomial growth in $\abs{x}$.

\item[\ttup{iii}]
We assume that the Neumann eigenvalues $\rho_n$ satisfy
\begin{equation}\label{EA3.1A}
\rhos\,\df\,\limsup_{n\to\infty}\,\rho_n\,>\,
\lim_{r\to\infty}\,\sup_{(x,\xi)\in B_r^c\times\Act}\,\rc(x,\xi)\,.
\end{equation}
\end{enumerate}
\end{assumption}

\Cref{A3.1} is enforced throughout the rest of the paper, unless mentioned
otherwise.
Part (i) of this assumption are the usual hypotheses that guarantee existence
and uniqueness of strong solutions to \cref{E-sde1} under any admissible control.

\begin{remark}
\Cref{EA3.1A} is a version of the near-monotone assumption,
which is often used in ergodic control problems (see \cite{ABG}).
This has the effect of penalizing instability, ensuring tightness
of laws for optimal controls.
There are two important cases where \cref{EA3.1A} is always satisfied.
First, when $-\rc$ is inf-compact.
In this case we have $\rhos\le \sup_{\Rd\times\Act} \rc$ and $\rhos>-\infty$, since the
Dirichlet eigenvalues which are a lower bound for $\rhos$ are increasing
as a function of the domain \cite[Lemma~2.1]{ABS-19}.
Second, when $\rc$ is positive and vanishes at infinity, and under some
stationary Markov control the process $\{X_t\}_{t\ge0}$ in \cref{E-sde1} is recurrent.
This can be established by comparing $\rho_n$ with the Dirichlet eigenvalue
on $B_n$ (see \cref{S3.2}), and using \cite[Theorems~2.6 and 2.7\,(ii)]{ABS-19}.
For related studies concerning the class of running reward functions vanishing
at infinity, albeit in the uncontrolled case, see
\cite{Ichihara-13b,Ichihara-15,ABS-19,Armstrong-09}.
See also \cite[Theorem~2.12]{AB-18b} which studies the Collatz--Wielandt formula
for the risk-sensitive minimization problem.
\end{remark}

Recall that $\Usm$ denotes the set of stationary Markov controls.
For $v\in\Usm$, we use the simplifying notation
\begin{equation*}
b_v(x) \,\df\, b\bigl(x,v(x)\bigr)\,,\qquad
\rc_v(x) \,\df\, \rc\bigl(x,v(x)\bigr)\,,
\end{equation*}
and define $\Lg_v$ analogously.

We next review some properties of eigenvalues of linear and semilinear operators
on $\Rd$.
For $f\in\Cc^2(\Rd)$ and $\psi\in\Sobl^{2,d}(\Rd)$, define
\begin{equation}\label{E-twist}
\widetilde\Lg^\psi_\xi f \,\df\,
\Lg_\xi f + \langle a\grad\psi,\grad f\rangle\,,
\end{equation}
with $\Lg_\xi$ as in \cref{E-gen}.
Let $v\in\Usm$.
Suppose that a positive function $\Psi\in\Sobl^{2,d}(\Rd)$ and $\lambda\in\RR$
solve the equation
\begin{equation}\label{E-eigen}
\Lg_v \Psi(x) + \rc_v(x) \Psi(x) \,=\, \lambda \Psi(x)\quad\text{a.e.\ } x\in\Rd\,.
\end{equation}
We refer to any such solution $(\Psi,\lambda)$ as an \emph{eigenpair of
the operator $\Lg_v+\rc_v$}, and we say that $\Psi$ is an eigenvector
with eigenvalue $\lambda$.
Note that by eigenvector we always mean a positive function.
Let $\psi=\log\Psi$.
We refer to the It\^o stochastic differential
equation
\begin{equation}\label{E-sde2}
\D \widetilde{X}_t \,=\, \bigl(b_v(\widetilde{X}_t)
+ a(\widetilde{X}_t)\grad\psi(\widetilde{X}_t)\bigr)
\,\D t + \upsigma (\widetilde{X}_t)\,\D W_t
\end{equation}
as the \emph{twisted} SDE, and to its solution
as the \emph{twisted} process corresponding to $\Psi$.
Clearly $\widetilde\Lg^\psi_v$ is the extended generator of \cref{E-sde2}.

We define the \emph{generalized principal eigenvalue} $\lambda_v=\lambda_v(\rc_v)$ of
the operator $\Lg_v + \rc_v$ by
\begin{equation}\label{E-princ1}
\lambda_v \,\df\,\inf\,\Bigl\{\lambda\in\RR\,
\colon \exists\, \phi\in\Sobl^{2,d}(\Rd),\ \phi>0, \
\Lg_{v}\phi + (\rc_{v}-\lambda)\phi\le 0 \text{\ a.e.\ in\ } \Rd\Bigr\}\,.
\end{equation}
A \emph{principal eigenvector} $\Psiv\in\Sobl^{2,d}(\Rd)$
is a positive solution of \cref{E-eigen} with $\lambda=\lambda_v$.
A principal eigenvector is also called a \emph{ground state}, and
we refer to the corresponding twisted SDE and twisted process as
a \emph{ground state SDE} and \emph{ground state process} respectively.
Unlike what is common in criticality theory, our definition of a
ground state does not require the minimal growth property of the principal eigenfunction
(see \cite{ABG-19}).

An easy calculation shows that any eigenpair
$(\Psi,\lambda)$ of $\Lg_v + \rc_v$ satisfies
\begin{equation}\label{E-eigenT}
\widetilde\Lg^\psi_v \Psi^{-1}(x) - \rc_v(x) \Psi^{-1}(x)
\,=\, -\lambda \Psi^{-1}(x)\quad\text{a.e.\ } x\in\Rd\,,
\end{equation}
with $\psi=\log\Psi$.
In other words, $(\Psi^{-1},-\lambda)$ is an eigenpair of $\widetilde\Lg^\psi_v-\rc_v$.
Note also that $(\psi,\lambda)$ is a solution to the `linear' eigenvalue equation
\begin{equation}\label{E-eigen2}
\widetilde\Lg^\psi_v \psi - \tfrac{1}{2}\abs{\upsigma\transp\nabla\psi}^2
+ \rc_v \,=\, \lambda\,,
\end{equation}
and that this equation can also be written as
\begin{equation}\label{E-eigen3}
\Lg_v \psi + \max_{y\in\Rd}\,\Bigl[\langle a y,\nabla\psi\rangle
 -\tfrac{1}{2}\abs{\upsigma\transp y}^2\Bigr]
+ \rc_v \,=\, \lambda\,.
\end{equation}

An extensive study of generalized principal eigenvalues
with applications to risk-sensitive control can be found in \cite{AB-18,ABS-19}.
In these papers, the `potential' $\rc_v$ is assumed to be bounded below in $\Rd$,
so the results cannot be quoted directly.
It is not our intention to reproduce all these results for potentials which
are bounded above, so we only focus on results that are needed later in this paper.
We only quote results in \cite{AB-18,ABS-19} which do not
depend on the assumption that $\rc_v$ is bounded below.
Generally speaking, caution should be exercised with arguments
in \cite{AB-18,ABS-19} that employ the Fatou lemma.
On the other hand, since $\rc$ usually appears in the exponent, invoking Fatou's lemma
hardly ever poses any problems.

Suppose that the twisted process in \cref{E-sde2} is regular, that is,
the solution exists for all times.
Then, an application of \cite[Lemma~2.3]{ABS-19} shows that
an eigenvector $\Psi$ has the stochastic representation
(semigroup property)
\begin{equation*}
\Psi(x) \,=\, \Exp^x_v\Bigl[\E^{\int_0^{t}[\rc_{v}(X_s)-\lambda]\, \D{s}}\,
\Psi(X_t)\Bigr]\,.
\end{equation*}

Recall that $\uuptau_r$ denotes the first hitting time of the ball $B_r$, for $r>0$.
We need the following lemma.

\begin{lemma}\label{L3.1}
We assume only \cref{A3.1}\,\textup{(i)--(ii)}.
The following hold.
\begin{enumerate}
\item[\ttup a]
If $(\Psi,\lambda)$ is an eigenpair of $\Lg_v + \rc_v$ under some $v\in\Usm$,
and the twisted process in \cref{E-sde2} is exponentially ergodic, then
 we have the stochastic representation
\begin{equation}\label{EL3.1A}
\Psi(x) \,=\, \Exp^x_v \Bigl[\E^{\int_0^{\uuptau_r}[\rc_{v}(X_s)-\lambda]\, \D{s}}\,
\Psi(X_{\uuptau_r})\,\Ind_{\{\uuptau_r<\infty\}}\Bigr]\quad\forall\,x\in\Bar{B}_r^c\,,
\ \forall\,r>0\,.
\end{equation}
In addition, $\lambda=\lambda_v$, the generalized principal
eigenvalue of $\Lg_v + \rc_v$, and the ground state $\Psi=\Psiv$ is unique up
to multiplication by a positive constant.

\item[\ttup b]
Any eigenpair $(\Psi,\lambda)\in\Sobl^{2,d}(\Rd)\times\Rd$
of $\Lg_v + \rc_v$ satisfying \cref{EL3.1A} is a principal eigenpair,
and $\lambda$ is a simple eigenvalue.
\end{enumerate}
\end{lemma}

\begin{proof}
Combining the proof of \cite[Theorem~2.2]{ABS-19} with
\cite[Theorem~3.1]{ABS-19}, we deduce that for every $r>0$, there exists a $\delta>0$
such that
\begin{equation}\label{PL3.1A}
\Exp^x_v\Bigl[\E^{\int_0^{\uuptau_r} [\rc_v(X_s)-\lambda + \delta]\,\D{s}}
\,\Ind_{\{\uuptau_r<\infty\}}\Bigr] \,<\, \infty\,,
\quad x\in B_r^c.
\end{equation}
Applying the It\^o formula to \cref{E-eigen} we obtain
\begin{equation}\label{PL3.1B}
\begin{aligned}
\Psi(x) &\,=\, \Exp^x_v \Bigl[\E^{\int_0^{t\wedge\uuptau_r\wedge\uptau_n}
[\rc_{v}(X_s)-\lambda]\, \D{s}}\,
\Psi(X_{t\wedge\uuptau_r\wedge\uptau_n})\Bigr]\\
&\,=\, \Exp^x_v \Bigl[\E^{\int_0^{\uuptau_r} [\rc_{v}(X_s)-\lambda]\, \D{s}}\,
\Psi(X_{\uuptau_r})\,\Ind_{\{\uuptau_r<t\wedge\uptau_n\}}\Bigr]\\
&\mspace{70mu}
+ \E^{-\delta t} \Exp^x_v \Bigl[\E^{\int_0^{t} [\rc_{v}(X_s)-\lambda+\delta]\, \D{s}}\,
\Psi(X_{t})\,\Ind_{\{t<\uuptau_r\wedge\uptau_n\}}\Bigr]\\
&\mspace{140mu}
+ \Exp^x_v \Bigl[\E^{\int_0^{\uptau_n} [\rc_{v}(X_s)-\lambda]\, \D{s}}\,
\Psi(X_{\uptau_n})\,\Ind_{\{\uptau_n<t\wedge\uuptau_r\}}\Bigr]\,.
\end{aligned}
\end{equation}
We study separately the three integrals on the right-hand side of \cref{PL3.1B},
which we denote as $\mathscr{J}_i$, $i=1,2,3$.
For the first integral we have
\begin{equation*}
\lim_{n\to\infty}\,\lim_{t\to\infty}\, \mathscr{J}_1
\,=\, \Exp^x_v \Bigl[\E^{\int_0^{\uuptau_r}[\rc_{v}(X_s)-\lambda]\, \D{s}}\,
\Psi(X_{\uuptau_r})\,\Ind_{\{\uuptau_r<\infty\}}\Bigr]
\end{equation*}
by monotone convergence. Note that the limit is also finite by \cref{PL3.1A}.

Let
$\widetilde\Prob^x_{\psi,v}$ and $\widetilde\Exp^x_{\psi,v}$ denote the
probability measure and expectation operator on the canonical space of the
twisted process in \cref{E-sde2} with initial condition $\Tilde{X}_0=x$.
Next, using again the technique in \cite[Theorem~2.2]{ABS-19},
we write
\begin{equation*}
\begin{aligned}
\mathscr{J}_2
&\,=\, \E^{-\delta t}\,\Exp^x_v\Bigl[\E^{\int_0^{t\wedge\uuptau_r\wedge\uptau_n}
[\rc_v(X_s)-\lambda + \delta]\,\D{s}}
\,\Psi(X_{t\wedge\uuptau_r\wedge\uptau_n})\,\Ind_{\{t<\uuptau_r\wedge\uptau_n\}}\Bigr]\\
&
\,\le\, \E^{-\delta t} \, \Exp^x_v \Bigl[\E^{\int_0^{t\wedge\uuptau_r\wedge\uptau_n}
[\rc_{v}(X_s)-\lambda+ \delta]\, \D{s}}\,
\Psi(X_{t\wedge\uuptau_r\wedge\uptau_n})\Bigr]\\
&
\,\le\, \E^{-\delta t}\,
\widetilde\Exp^x_{\psi,v}\Bigl[\E^{\delta(t\wedge\uuptau_r\wedge\uptau_n)}\Bigr]
\,\le\, \E^{-\delta t}\,
\widetilde\Exp^x_{\psi,v}\bigl[\E^{\delta\uuptau_r}\bigr]\,,
\end{aligned}
\end{equation*}
where in the second inequality we apply \cite[Lemma~2.3]{ABS-19}.
Thus, $\mathscr{J}_2$ vanishes as $t\to\infty$.

Concerning $\mathscr{J}_3$, using monotone convergence, we obtain
\begin{equation}\label{PL3.1C}
\lim_{t\to\infty}\, \mathscr{J}_3
\,=\, \Exp^x_v \Bigl[\E^{\int_0^{\uptau_n} [\rc_{v}(X_s)-\lambda]\, \D{s}}\,
\Psi(X_{\uptau_n})\,\Ind_{\{\uptau_n<\uuptau_r\}}\Bigr]
\,\le\, \Psi(x)\,\widetilde\Prob^x_{\psi,v}\bigl(\uptau_n<\uuptau_r)\,.
\end{equation}
where the inequality follows from the proof of \cite[Lemma~2.3]{ABS-19}.
In turn, the right-hand side of \cref{PL3.1C}
vanishes as $n\to\infty$, since the twisted process is geometrically ergodic.
This completes the proof of \cref{EL3.1A}.

Suppose that a positive $\phi\in\Sobl^{2,d}(\Rd)$ and $\Hat\lambda\le\lambda$ solve
\begin{equation*}
\Lg_v \phi(x) + \rc_v(x) \phi(x) \,\le\, \Hat\lambda \phi(x)
\quad\text{a.e.\ } x\in\Rd\,.
\end{equation*}
An application of It\^o's formula and Fatou's lemma then shows that
\begin{equation}\label{PL3.1D}
\phi(x) \,\ge\,
\Exp^x_v \Bigl[\E^{\int_0^{\uuptau_r}[\rc_{v}(X_s)-\Hat\lambda]\, \D{s}}\,
\phi(X_{\uuptau_r})\,\Ind_{\{\uuptau_r<\infty\}}\Bigr]\qquad\forall\, x\in\Bar{B}_r^c\,,
\ \ \forall\,r>0\,.
\end{equation}
\Cref{EL3.1A,PL3.1D} imply that if we scale $\phi$ by multiplying it with
a positive constant until it touches $\Psi$ at one point from above,
the function $\frac{\phi}{\Psi}$ attains its minimum value of $1$ at some
point in $\Bar{B}_r$.
A standard calculation shows that
\begin{equation*}
\widetilde\Lg^\psi_v \bigl(\tfrac{\phi}{\Psi}\bigr)(x) \,\le\, (\Hat\lambda-\lambda)
\bigl(\tfrac{\phi}{\Psi}\bigr)(x)\,.
\end{equation*}
Thus, $\frac{\phi}{\Psi}$ must equal a constant by the strong maximum principle,
which implies that $\Hat\lambda=\lambda$.
This of course means that $\lambda=\lambda_v$.
Uniqueness of $\Psiv$ is evident from the preceding argument.
This completes the proof of part (a).

Part (b) is evident from the preceding paragraph.
This completes the proof.
\end{proof}

\subsection{The Bellman equation in \texorpdfstring{$\Rd$}{}}

Recall the solution $(V_r,\rho_r)$ of \eqref{E-HJBN}, the definition of $\rhos$ in
\cref{EA3.1A}, and the definition of $\cG$ in \cref{EcG}.
We define
\begin{equation}\label{E-princ2}
\lambda_*\,\df\,\inf\,\Bigl\{\lambda\in\RR\,
\colon \exists\, \phi\in\Sobl^{2,d}(\Rd),\ \phi>0, \
\cG\phi -\lambda\phi\le 0 \text{\ a.e.\ in\ } \Rd\Bigr\}\,.
\end{equation}
Recall the definitions of $\sA$ and $\sR$ in \cref{EsA,EsR}.
Note that if $(\Phi,\lambda)$ is an eigenpair of $\cG$,
then similarly to \cref{E-eigen3}, we have
\begin{equation}\label{E-eigen4}
\adjustlimits\max_{\xi\in\Act}
\max_{y\in\Rd}\,\bigl[\sA\varphi(x,\xi,y) + \sR(x,\xi,y)\bigr]
\,=\,\lambda\,,
\end{equation}
with $\varphi=\log\Phi$.

\begin{theorem}\label{T3.1}
There exists $\Phis\in\Cc^2(\Rd)$ satisfying
\begin{equation}\label{ET3.1A}
\max_{\xi\in\Act}\,\bigl[\Lg_\xi \Phis(x) + \rc(x,\xi) \Phis(x) \bigr] \,=\,
\rhos \Phis(x) \quad\forall\,x\in\Rd\,,
\end{equation}
and the following hold:
\begin{enumerate}
\item[\ttup a]
The function $\Phis^{-1}$ is inf-compact.

\item[\ttup b]
If $v_*$ is an a.e.\ measurable selector from the maximizer of \cref{ET3.1A},
then, the diffusion with extended generator $\widetilde\Lg_{v_*}^{\varphis}$,
as defined in \cref{E-twist}, is exponentially ergodic and satisfies
\begin{equation}\label{ET3.1B}
\widetilde\Lg_{v_*}^{\varphis}\Phis^{-1}(x) \,=\,
\bigl(\rc_{v_*}(x)-\rhos\bigr)\, \Phis^{-1}(x)\,,
\end{equation}
with $\varphis\df\log\Phis$.

\item[\ttup c] $\rhos = \lambda_*$.

\item[\ttup d]
$\rho_n\to\rhos$ and $V_n\to\Phis$ as $n\to\infty$
uniformly on compact sets, and
the solution $\Phis$ to \cref{ET3.1A} is unique up to a scalar multiple, and
satisfies
\begin{equation}\label{ET3.1C}
\begin{aligned}
\Phis(x) \,\ge\, \Exp^x_v \Bigl[\E^{\int_0^{\uuptau_r}[\rc_{v}(X_s)-\rhos]\, \D{s}}\,
\Phis(X_{\uuptau_r})\,\Ind_{\{\uuptau_r<\infty\}}\Bigr]\qquad\forall\,x\in\Bar{B}_r^c\,,
\end{aligned}
\end{equation}
for all $r>0$,
and for all $v\in\Usm$, with equality if and only if $v$ is an a.e.\ measurable selector
from the maximizer in \cref{ET3.1A}.
\end{enumerate}
\end{theorem}

\begin{proof}
Using Theorem~\ref{T2.1} and \eqref{E-JQ}-\eqref{E-J*}, it follows
that $\rho_n \le \sup_{\Rd\times\Act} \rc$, and this combined with \cref{A3.1}\,(iii)
shows that $\{\rho_n\}$ converges along some subsequence
$\{n_k\}_{k\in\NN}\subset\NN$ to $\rhos$.
Therefore, the convergence of $V_{n_k}$ along some further subsequence
$\{ n_k'\}\subset\{n_k\}$ to a $\Phis$ satisfying \cref{ET3.1A}
follows as in the proof of \cite[Lemma~2.1]{Biswas-11a}.

We now turn to part (a). Here in fact we show that $-\abs{\varphis}$ has at least
logarithmic growth in $\abs{x}$.
Let $\delta\in(0,1)$ be a constant such that
$\rhos-\rc(x,\xi) >4\delta$ for all $x$ outside some compact set in $\Rd$.
Consider a function of the form
$\phi(x) = \bigl(1 + \abs{x}^2\bigr)^{-\theta}$, with $\theta>0$.
By \cref{E-growth}, there exists $\theta>0$ and $r_\circ>0$ such that
\begin{equation}\label{ER3.2A}
\max\,\bigl(\Lg_\xi \phi(x), \babs{\upsigma\transp(x)\nabla\phi(x)}\bigr)
\,\le\, \delta\phi(x)\qquad\forall\,(x,\xi)\in B_{r_\circ}^c\times\Act\,.
\end{equation}
We fix such a constant $\theta$.
We restrict our attention to solutions $(V_n,\rho_n)$ of \cref{E-HJBN} over
an increasing sequence in $\NN$, also denoted as $\{n\}$,
 such that $\rho^{}_n$ converges to $\rhos$.
It is clear then that we may enlarge the radius $r_\circ$, if needed, so that
\begin{equation}\label{ER3.2B}
\rho^{}_n-\rc(x,\xi)\,>\, 3\delta\,\qquad \forall\,(x,\xi)\in B_{r_\circ}^c\times\Act\,,
\ \text{and\ } n\ge r_\circ\,.
\end{equation}

Next, let $\Breve\chi\colon\RR\to(0,\infty)$ be a convex function in $\Cc^2(\RR)$
such that $\Breve\chi(t)=t$ for $t\ge2$, and $\Breve\chi(t)$ is constant
and positive for $t\le 1$.
This can be chosen so that $\Breve\chi''<2$ and $\sup_{t>0}\, t\Breve\chi''(t)<2$.
Such a function can be constructed by requiring, for example, that
$\Breve\chi''(t) = 6 (2-t)(t-1)$ for $t\in[1,2]$,
from which we obtain
$\Breve\chi(t) = -\frac{1}{2} t^4 + 3 t^3 -6t^2 + 5t$ for $t\in[1,2]$.
A simple calculation shows that $\Breve\chi(1) = \frac{3}{2}$.
Note that $\Breve\chi(t)-t\Breve\chi'(t)\ge0$ for all $t>0$ by convexity.
Let $\Breve\chi_\epsilon(t) \df \epsilon\Breve\chi\bigl(\nicefrac{t}{\epsilon}\bigr)$
for $\epsilon>0$.
Then
\begin{equation}\label{ER3.2C}
\Breve\chi_\epsilon(t)-t\Breve\chi'_\epsilon(t)\,\ge\,0\,,\quad\text{and\ \ }
t \Breve\chi''_\epsilon(t)\,<\,2 \qquad\forall\,t>0\,.
\end{equation}
Using \cref{ER3.2A,ER3.2B,ER3.2C},
we obtain 
\begin{equation}\label{ER3.2D}
\begin{aligned}
\Lg_\xi \Breve\chi_\epsilon\bigl(\phi(x)\bigr)
&+ \bigl(\rc(x,\xi)-\rho^{}_n\bigr) \Breve\chi_\epsilon\bigl(\phi(x)\bigr)\\
&\,\le\, -3\delta \Breve\chi_\epsilon\bigl(\phi(x)\bigr) 
+ \Breve\chi'_\epsilon\bigl(\phi(x)\bigr)\,\Lg_\xi \phi(x)
+\frac{1}{2}\Breve\chi''_\epsilon\bigl(\phi(x)\bigr) 
\abs{\upsigma\transp(x)\nabla\phi(x)}^2\\
&\le -3\delta \Breve\chi_\epsilon\bigl(\phi(x)\bigr)
+\delta \phi(x)\,\Breve\chi'_\epsilon\bigl(\phi(x)\bigr)
+ \frac{1}{2} \delta^2 \bigl(\phi(x)\bigr)^2\Breve\chi''_\epsilon\bigl(\phi(x)\bigr)\\
&\le -\delta \Breve\chi_\epsilon\bigl(\phi(x)\bigr)\,.
\end{aligned}
\end{equation}
For the last inequality in \cref{ER3.2D}, we use
the properties $\Breve\chi_\epsilon(\phi)
\ge \phi\,\Breve\chi_\epsilon'(\phi)$
and $\phi\,\Breve\chi_\epsilon''(\phi)<2$
from \cref{ER3.2C},
that the fact that $\Breve\chi_\epsilon(\phi)\ge\phi$ and $\delta<1$.
Note that, due to radial symmetry, the support of $\Breve\chi'_\epsilon\comp\phi$
is a ball of the form $B_{R_\epsilon}$, with $\epsilon\mapsto R_\epsilon$ an
nonincreasing continuous function with $R_\epsilon\to\infty$ as $\epsilon\searrow0$.
Recall the functions $V_n$ in \cref{E-HJBN}.
Select $\epsilon$ such that $R_\epsilon=n>r_\circ$.
Scale $V_n$ until it touches $\Breve\chi_\epsilon\comp\phi$ at some point $\Hat{x}$
from below.
Here, $\Breve\chi_\epsilon\comp\phi$ denotes the composition of
$\Breve\chi_\epsilon$ and $\phi$.
Let $v^{}_n$ be a measurable selector from the minimizer in \cref{E-HJBN},
and define $h_n \df \Breve\chi_\epsilon\comp\phi - V_n$.
Then, by \cref{E-HJBN,ER3.2D}, we have
\begin{equation*}
\Lg_{v^{}_n} h_n(x)
+ \bigl(\rc_{v^{}_n}(x)-\rho^{}_n\bigr) h_n(x) \,<\,0
\qquad\forall\,x\in\Rd\,,
\end{equation*}
and $\langle\nabla h_n,\gamma\rangle =0$ on $\partial B_n$,
since the gradient of $\Breve\chi_\epsilon\comp\phi$ vanishes on
$\partial B_{R_\epsilon}$.
It follows by the strong maximum principle that $\Hat{x}$ cannot lie in the
$B_{n}\setminus B_{r_\circ}$.
Thus $h_n>0$ on this set.
This implies that $\Hat{x}$ cannot lie on $\partial B_n$ either,
without contradicting the Hopf boundary point lemma.
Thus $\Hat{x}\in B_{r_\circ}$.
This however shows by taking limits as $\epsilon\searrow0$, and employing
the Harnack inequality which asserts that $V_n(x)\le C_{\mathsf{H}} V_n(y)$
for all $x,y\in B_{r_\circ}$ for some constant $C_{\mathsf{H}}$,
that $\Phis\le C\phi$ for some constant $C$.
This proves part (a).

\Cref{ET3.1B} follows by \cref{E-eigenT}.
Since $\Phis^{-1}$ is inf-compact and the right hand side of
\cref{ET3.1B} is negative and bounded away from zero outside a compact set
by \cref{A3.1}\,(iii),
the associated diffusion is ergodic \cite[Theorem~4.1]{Ichihara-13b}.
In turn, the Foster--Lyapunov equation in \cref{ET3.1B} shows
that the diffusion is exponentially ergodic \cite{MeTw}.
This proves part (b).

Moving to the proof of part (c), suppose
that for some $\rho\le\rhos$ we have
\begin{equation}\label{PT3.1E}
\max_{\xi\in\Act}\,\bigl[\Lg_\xi \phi(x) + \rc(x,\xi) \phi(x) \bigr]
\,\le\, \rho\,\phi(x)\,.
\end{equation}
Evaluating this equation at measurable selector $v_*$ from the maximizer
of \cref{ET3.1A},
and following the argument in the proof
of \cref{L3.1} we obtain $\rho=\rhos$ and $\phi=\Phis$.
This also shows that $\rhos\ge\lambda_*$ by the definition in \cref{E-princ2},
and thus we have equality by \cref{ET3.1A}.

In order to prove part (d), suppose that $\rho_n\to\rho\le\rhos$ along
some subsequence.
Taking limits along perhaps a further subsequence, we obtain a positive
function $\phi\in\Cc^2(\Rd)$ that satisfies \cref{PT3.1E} with equality.
Thus $\rho=\rhos$ and and $\phi=\Phis$ by part (c).
The stochastic representation in \cref{ET3.1C} follows as in the proof of \cref{L3.1}.
This completes the proof.
\end{proof}

\subsection{Dirichlet eigenvalues and the risk-sensitive value}\label{S3.2}

In this section we first show that the problem in $\Rd$ can also be approached
by using Dirichlet eigensolutions.
The main result is \cref{T3.2}, which establishes that
$\rhos$ equals the risk-sensitive value $J_*$,
and the usual verification of optimality criterion.

We borrow some results from \cite{BNV-94,Berestycki-15}.
These can also be found in \cite[Lemma~2.2]{AB-18}, and are summarized
as follows:
Fix any $v\in\Usm$.
For each $r\in(0,\infty)$ there exists a unique pair
$(\Psi_{\mspace{-2mu}v,r},\lambda_{v,r})
\in\bigl(\Sob^{2,p}(B_r)\cap\Cc(\Bar{B}_r)\bigr)\times\RR$,
for any $p>d$, satisfying
$\Psi_{\mspace{-2mu}v,r}>0$ on $B_r$, $\Psi_{\mspace{-2mu}v,r}=0$ on
$\partial B_r$, and $\Psi_{\mspace{-2mu}v,r}(0)=1$,
which solves
\begin{equation}\label{E-Leigen}
\Lg_v \Psi_{\mspace{-2mu}v,r}(x) + \rc_v(x)\,\Psi_{\mspace{-2mu}v,r}(x)
\,=\, \lambda_{v,r}\,\Psi_{\mspace{-2mu}v,r}(x)
\qquad\text{a.e.\ }x\in B_r\,.
\end{equation}
Moreover, the solution has the following properties:
\begin{enumerate}
\item[(i)]
The map $r\mapsto\lambda_{v,r}$ is continuous and strictly increasing.

\item[(ii)]
In its dependence on the function $\rc_v$,
 $\lambda_{v,r}$ is nondecreasing, convex, and Lipschitz continuous
 (with respect to the $\Lp^{\infty}$ norm) with Lipschitz constant $1$.
In addition, if $\rc_v\lneqq \rc_v'$ then $\lambda_{v,r}(\rc_v)<\lambda_{v,r}(\rc_v')$.
\end{enumerate}
We refer to $\lambda_{v,r}$ and $\Psi_{\mspace{-2mu}v,r}$ as the
(Dirichlet) eigenvalue
and eigenfunction, respectively, of the operator $\Lg_v + \rc_v$ on
$B_r$.

Recall the definition of $\cG$ in \cref{EcG}.
Based on the results in \cite{Quaas-08}, there exists a unique pair
$(\Psi_{\mspace{-2mu}*,r},\lambda_{*,r})\in
\bigl(\Cc^2(B_r)\cap\Cc(\Bar{B}_r)\bigr)\times\RR$, satisfying
$\Psi_{\mspace{-2mu}*,r}>0$ on $B_r$, $\Psi_{\mspace{-2mu}*,r}=0$ on
$\partial B_r$, and $\Psi_{\mspace{-2mu}*,r}(0)=1$, which solves
\begin{equation}\label{E-Geigen}
\cG \Psi_{\mspace{-2mu}*,r}(x)
\,=\, \lambda_{*,r}\,\Psi_{\mspace{-2mu}*,r}(x)
\qquad\forall\,x\in B_r\,,
\end{equation}
and properties (i)--(ii) above hold for $\lambda_{*,r}$.
Also recall the definitions of the generalized principal eigenvalues
in \cref{E-princ1,E-princ2}, and $\rho_r$ defined in \cref{E-HJBN}.

\begin{lemma}\label{L3.2}
The following hold:
\begin{enumerate}
\item[\ttup i]
For $r>0$, we have $\lambda_{v,r}\le \lambda_{*,r}$ for all $v\in\Usm$, and
$\lambda_{*,r}< \rho_r$.

\item[\ttup {ii}]
$\lim_{r\to\infty}\, \lambda_{v,r}=\lambda_v$ for all $v\in\Usm$, and
$\lim_{r\to\infty}\, \lambda_{*,r}=\lambda_*$.
\end{enumerate}
\end{lemma}

\begin{proof}
Part (i) is a straightforward application of the strong maximum principle.
By \cref{E-gen,E-Geigen} we have
\begin{equation}\label{PL3.2A}
\Lg_v \Psi_{\mspace{-2mu}*,r}(x) + \rc_v(x)\,\Psi_{\mspace{-2mu}*,r}(x)
\,\le\, \lambda_{*,r}\,\Psi_{\mspace{-2mu}*,r}(x)
\qquad\text{a.e.\ }x\in B_r\,.
\end{equation}
Let $r'<r$, and suppose that $\lambda_{v,r'}\,\ge\, \lambda_{*,r}$.
Scale $\Psi_{\mspace{-2mu}v,r'}$ so that it touches $\Psi_{\mspace{-2mu}*,r}$
at one point from below in $B_{r'}$.
Then $\Psi_{\mspace{-2mu}*,r}-\Psi_{\mspace{-2mu}v,r'}$ is nonnegative, and
by \cref{E-Leigen,PL3.2A} it satisfies
\begin{equation*}
\begin{aligned}
\Lg_v(\Psi_{\mspace{-2mu}*,r}&-\Psi_{\mspace{-2mu}v,r'})
-\bigl(\rc_v-\lambda_{*,r}\bigr)^{-} (\Psi_{\mspace{-2mu}*,r}-\Psi_{\mspace{-2mu}v,r'})\\
&\,=\, -\bigl(\rc_v-\lambda_{*,r}\bigr)^{+}
(\Psi_{\mspace{-2mu}*,r}-\Psi_{\mspace{-2mu}v,r'})
 - \bigl(\lambda_{v,r}-\lambda_{*,r}\bigr)\Psi_{\mspace{-2mu}v,r'}
\,\le\, 0
\quad\text{a.e.\ on\ } B_{r'}\,.
\end{aligned}
\end{equation*}
This however implies that $\Psi_{\mspace{-2mu}*,r}=\Psi_{\mspace{-2mu}v,r'}$ on $B_{r'}$
which is a contradiction.
Hence $\lambda_{v,r'}\,<\, \lambda_{*,r}$ for all $r'<r$ and
the inequality $\lambda_{v,r}\le \lambda_{*,r}$ follows by the continuity
of $r\mapsto\lambda_{v,r}$.
Following the same method, with $r'=r$, we obtain $\lambda_{*,r}< \rho_r$.

Part (ii) follows by \cite[Lemma~2.2\,(ii)]{ABS-19}.
\end{proof}

Recall the definitions in \cref{E-JQ,E-J*}, and let
\begin{equation*}
J^x_\xi\,=\, J^x_\xi(\rc) \,\df\, J^x_\xi(\rc;\Rd)\,,
\end{equation*}
and similarly for $J^x_*$ and $J_*$.
Also, recall that
\begin{equation*}
J^x_v \,=\, J^x_v(\rc) \,=\, \liminf_{T\to\infty}\, \frac{1}{T}\,
\log \Exp^x_v \Bigl[\E^{\int^T_0 \rc_v(X_t)\,\D t} \Bigr]\,,
\quad x\in \Rd\,,\ v\in\Usm\,.
\end{equation*}
The theorem that follows concerns the equality $\lambda_*=J_*$.
Recall the definition in \cref{EA3.1A}.

\begin{theorem}\label{T3.2}
We have
$\lambda_*=\rhos =J_*$.
In addition, $J^x_v=J_*$ if and only if
$v$ is an a.e.\ measurable selector from the maximizer of \cref{ET3.1A}.
\end{theorem}

\begin{proof}
We already have $\rhos=\lambda_*$ from \cref{T3.1}. This also gives
\begin{equation*}
\rhos \,\le\, J^x_{v_*}(\rc) \,\le\, J_*\,.
\end{equation*}
Choose $R>0$ such that $\rhos>\sup_{B^c_R\times\Act}\,\rc$.
This is possible by \cref{EA3.1A}.
Let $\delta>0$ be given, and select a smooth, non-negative cut-off function $\chi$ that
vanishes in $B_R$ and equals to $1$ in $B_{R+1}^c$.
Let $\Psi=\Phis+\varepsilon \chi$, and select $\epsilon>0$ small enough so that
\begin{equation*}
\epsilon\, \bigl(\cG\chi(x) -\rhos\chi(x)\bigr)\,\le\,
\delta\, \Phis(x)\qquad\forall\,x\in\bar{B}_{R+1}\,.
\end{equation*}
This is clearly possible since $\Phis$ is positive and
\begin{equation*}
\cG\chi(x) -\rhos\chi(x) \,=\,
\max_{\xi\in\Act}\, (\rc(x,\xi)-\rhos) \chi(x) \,\le\, 0
\qquad\forall\,x\in B_{R+1}^c\,.
\end{equation*}
We have
\begin{equation}\label{PT3.2A}
\cG \Psi(x) - (\rhos+\delta)\Psi(x)
\,\le\, (\cG-\rhos)\Phis(x) +\epsilon\,(\cG-\rhos)\chi(x)
- \delta\,\Psi(x)\,\le\,0\quad\forall\,x\in\Rd\,.
\end{equation}
Since $\Psi$ is bounded below away from zero, a standard use of
It\^o's formula and the Fatou lemma applied to \cref{PT3.2A} shows that
$J^x_\xi\le \rhos+\delta$ for all $\xi\in\Uadm$.
Since $\delta$ is arbitrary this implies $\rhos\ge J_*$, and hence
we must have equality.
This also shows that every a.e.\ measurable selector from the maximizer of \cref{ET3.1A}
is optimal.

Next, for $v\in\Usm$, let $(\lambda_v,\Psiv)$ be an eigenpair,
obtained as a limit of
Dirichlet eigenpairs $\bigl\{(\lambda_{v,n},\Psi_{\mspace{-2mu}v,n})\bigr\}_{n\in\NN}$,
with $\Psi_{\mspace{-2mu}v,n}(0)=1$, along some subsequence (see \cref{L3.2}).
Let $\nu\in[-\infty,\infty)$ be defined by
\begin{equation*}
\nu \,\df\, \lim_{r\to\infty}\,\sup_{(x,\xi)\in B_r^c\times\Act}\,\rc(x,\xi)\,.
\end{equation*}
First suppose that $\lambda_v>\nu$.
Then, using the the argument in the preceding paragraph, together with the fact that
$\lambda_v\le J^x_v$, we deduce that $\lambda_v= J^x_v$ for all $x\in\Rd$.
Thus if $v\in\Usm$ is optimal, we must have $\lambda_v=\rhos$.
This implies that we can select a ball $\sB$ such that
\begin{equation*}
\lambda_{v,n}-\sup_{(x,\xi)\in \sB^c\times\Act}\,\rc(x,\xi)\,>\, 0
\end{equation*}
for all sufficiently large $n$.
Let $\uuptau= \uptau(\sB^c)$.
By \cite[Lemma~2.10\,(i)]{AB-18}, we have the stochastic representation
\begin{equation*}
\Psi_{\mspace{-2mu}v,n}(x)\,=\, \Exp^x_v \Bigl[\E^{\int_{0}^{\uuptau}
[\rc_v(X_{t})-\lambda_{v,n}]\,\D{t}}\, \Psi_{\mspace{-2mu}v,n}(X_{\uuptau})\,
\Ind_{\{\uuptau<\uptau_{n}\}}\Bigr]\qquad\forall\,
x\in B_{n}\setminus\Bar{\sB}\,.
\end{equation*}

Next we show that that $\Psiv$ vanishes at infinity by
using the argument in the proof of \cref{T3.1}.
The analysis is simpler here.
Selecting the same function $\phi$ as in the proof of \cref{T3.1},
there exists $R>0$ such that
\begin{equation*}
\Lg_v \phi(x) +\rc_v(x) \phi(x) \,\le\, \lambda_v \phi(x)\qquad\forall\,
x\in B_R^c\,.
\end{equation*}
Since $\Psi_{\mspace{-2mu}v,n}(0)=1$, employing the Harnack inequality we scale
$\phi$ so that $\phi>\Psi_{\mspace{-2mu}v,n}$ on $B_R$ for all $n>R$.
The strong maximum principle then shows that $\Psi_{\mspace{-2mu}v,n}<\phi$ on $\Rd$.

Thus $\Psiv^{-1}$ is inf-compact, which together with
the Lyapunov equation $\widetilde\Lg^{\psi^{}_v}_{v}\Psiv^{-1}
= \bigl(\rc_v-\rhos)\Psiv^{-1}$ imply that
the ground state process is exponentially ergodic.
By \cref{L3.1}, we then have
\begin{equation}\label{PT3.2B}
\Psiv(x)\,=\, \Exp^x_v \Bigl[\E^{\int_{0}^{\uuptau}
[\rc_v(X_{t})-\rhos]\,\D{t}}\, \Psiv(X_{\uuptau})\,
\Ind_{\{\uuptau<\infty\}}\Bigr]\qquad\forall\, x\in \Bar{\sB}^c\,.
\end{equation}
On the other hand, it holds that
$\Lg_v \Phis + \rc_v \Phis \le \rhos \Phis$,
which implies that
\begin{equation}\label{PT3.2D}
\begin{aligned}
\Phis(x) \,\ge\, \Exp^x_v \Bigl[\E^{\int_0^{\uuptau}[\rc_{v}(X_s)-\rhos]\, \D{s}}\,
\Phis(X_{\uuptau})\,\Ind_{\{\uuptau<\infty\}}\Bigr]\,.
\end{aligned}
\end{equation}
Comparing the functions in \cref{PT3.2B,PT3.2D} using the strong maximum principle,
as done in the proof of \cref{L3.1}, we deduce that $\Psiv=\Phis$.
Thus $v$ is a measurable selector from the maximizer of \cref{ET3.1A}.

It remains to address the case $\lambda_v\le\nu$.
By \cite[Corollary~3.2]{ABG-19} there exists a positive constant $\delta$ such that
$\lambda_v(\rc_v+\delta\Ind_{B_1})>\nu$, and $\lambda_v(\rc_v+\delta\Ind_{B_1})<\rhos$.
Thus repeating the above argument we obtain
\begin{equation*}
\rhos\,>\,\lambda_v(\rc_v+\delta\Ind_{B_1})\,=\,
\liminf_{T\to\infty}\, \frac{1}{T}\,
\log \Exp^x_v \Bigl[\E^{\int^T_0 [\rc_v(X_t)+\delta\Ind_{B_1}(X_t)]\,\D t} \Bigr]
\,\ge\, J^v_x\qquad
\forall\,x\in\Rd\,.
\end{equation*}
Therefore, $v$ cannot be optimal.
This completes the proof.
\end{proof}

\section{The variational formula on \texorpdfstring{$\Rd$}{}}\label{S4}

In this section we establish the variational formula on $\Rd$.
As mentioned in \cref{S1.1}, the function $\cH$ in \cref{Eentropy}
plays a very important role in the analysis.
To explain how this function arises, let $\Prob^{x,t}_v$ denote the probability
measure on the canonical path space $\{X_s\colon 0\le s\le t\}$ of the diffusion
\cref{E-sde1} under a control $v\in\Usm$, and 
$\widetilde\Prob^{x,t}_v$ the analogous probability measure
corresponding to the diffusion
\begin{equation*}
\D \widetilde{X}_t \,=\,
\bigl(b_v(\widetilde{X}_t)+a(x)\nabla\varphis(\widetilde{X}_t)\bigr)\,\D t
+ \upsigma (\widetilde{X}_t)\,\D \widetilde{W}_t\,,
\end{equation*}
with $\varphis$ as in \cref{T3.1}.
By the Cameron--Martin--Girsanov theorem we obtain
\begin{equation*}
\frac{\D\mathbb{P}^{x,t}_v}{\D \widetilde\Prob^{x,t}_v} \,=\,
\exp\biggl( -\int_0^{t} \bigl\langle \grad\varphis(\widetilde{X}_s),
\upsigma(\widetilde{X}_s) \D{\widetilde{W}_s}\bigr\rangle
-\frac{1}{2}\int_0^{t}
\babs{\upsigma\transp(\widetilde{X}_s)
\grad \varphis(\widetilde{X}_s)}^2\,\D{s}\biggr)\,.
\end{equation*}
Thus, the \emph{relative entropy}, or Kullback--Leibner
divergence between $\widetilde\Prob^{x,t}_v$ and $\Prob^{x,t}_v$ takes the form
\begin{equation*}
D_{\mathsf{KL}}\bigl(\widetilde\Prob^{x,t}_v \bigm\| \Prob^{x,t}_v\bigr)
\,=\, -\int \log\biggl(\frac{\D \mathbb{P}^{x,t}_v}{\D \widetilde\Prob^{x,t}_v}\biggr)\,
\D \widetilde\Prob^{x,t}_v
\,=\, \frac{1}{2}\,\widetilde\Exp^{x,t}_v\biggl[
\int_0^{t} \babs{\upsigma\transp(\widetilde{X}_s)
\grad \varphis(\widetilde{X}_s)}^2\,\D{s}\biggr]\,.
\end{equation*}
Dividing this by $t$, and letting $t\searrow0$, we see that
$\cH$ is the \emph{infinitesimal relative entropy rate}.

Recall from \cref{S1.1} the definition $\cZ \df \Rd\times\Act\times\Rd$, and
the use of the single variable $z=(x,\xi,y)\in\cZ$ in the interest of notational
simplicity.
Also recall the definitions in \cref{Eeom,Emufinite}.
Recall the definitions in \cref{EsA,EsR}.
In analogy to \cref{E-F}, we define
\begin{equation*}
F(g,\mu) \,\df\, \int_{\cZ} \bigl(\sA g(z)+\sR(z)\bigr)\,\mu(\D{z})
\quad \text{for\ } g \in \Cc^{2}(\Rd)\text{\ and\ } \mu\in\cP(\cZ)\,.
\end{equation*}

The following result plays a central role in this paper.

\begin{proposition}\label{P4.1}
We have
\begin{equation}\label{EP4.1A}
\rhos \,=\, \max_{\mu\in\eom_{\sA}\cap\cPs(\cZ)}\,\int_{\cZ} \sR(z)\,\mu(\D{z})
\,=\, \adjustlimits\sup_{\mu\in\cPs(\cZ)}\inf_{g \in \Cc^2_c(\Rd)}\,F(g,\mu) \,.
\end{equation}
In addition, if $\eom_{\sA}\cap\cPo(\cZ)\subset\cPs(\cZ)$, then
$\cPs(\cZ)$ may be replaced by $\cP(\cZ)$ in \cref{EP4.1A}.
\end{proposition}

In the proof of \cref{P4.1} and elsewhere in the paper we use
a cut-off function $\chi$ defined as follows (compare this with the function
$\Breve\chi$ in the proof of \cref{T3.1}).

\begin{definition}\label{D4.1}
Let $\chi\colon\RR\to\RR$ be a smooth convex function such that
$\chi(s)= s$ for $s\ge0$, and $\chi(s) = -1$ for $s\le -2$.
Then $\chi'$ and $\chi''$ are nonnegative and the latter is supported
on $(-2,0)$.
It is clear that we can choose $\chi$ so that $\chi''< 1$.
We scale this function by defining
$\chi^{}_t(s) \df -t + \chi(s+t)$ for $t\in\RR$.
Thus $\chi^{}_t(s)=s$ for $s\ge -t$, and $\chi^{}_t(s) = -t-1$ for $s\le -t-2$.
Observe that if $-f$ is an inf-compact function
then $\chi_t^{}(f)+t+1$ is compactly supported by the definition of $\chi$.
\end{definition}

\begin{proof}[Proof of \cref{P4.1}]
We start with the first equality in \cref{EP4.1A}.
By \cref{E-eigen2}, we have
\begin{equation}\label{PP4.1A}
\widetilde\Lg_{v_*}^{\varphis}\varphis(x)
+\rc_{v_*}(x) -\cH(x) \,=\, \rhos\,.
\end{equation}
As shown in \cref{T3.1} the twisted process $\Tilde{X}$ with extended
generator $\widetilde\Lg_{v_*}^{\varphis}$ is exponentially ergodic.
Let $\eta_{v_*}$ denote its invariant probability measure.
Since $\frac{\abs{\varphis}}{\Phis^{-1}}$ vanishes at infinity,
and $\Phis^{-1}$ is a Lyapunov function by \cref{ET3.1B},
it then follows from \cref{PP4.1A}, by using
the It\^o formula and applying \cite[Lemma~3.7.2\,(ii)]{ABG}, that
\begin{equation}\label{PP4.1B}
\rhos \,=\, \int_\Rd \bigl(\rc_{v_*}(x)- \cH(x)\bigr)\,\eta_{v_*}(\D{x})
\,=\, \int_\Rd \sR\bigl(x,v_*(x), \grad\varphis(x)\bigr)\,\eta_{v_*}(\D{x})\,.
\end{equation}

Next, we show that
\begin{equation}\label{PP4.1C}
\rhos \,\ge\, \int_{\cZ} \sR(z)\,\mu(\D{z})
\quad\forall\, \mu\in\eom_{\sA}\cap\cPs(\cZ)\,.
\end{equation}
We write \cref{ET3.1A} as
\begin{equation*}
\max_{\xi\in\Act}\,\Bigl[\Lg_\xi \varphis(x)
+\tfrac{1}{2} \babs{\upsigma\transp(x)\grad\varphis(x)}^2
+ \rc(x,\xi) \Bigr] \,=\,\rhos \quad\forall\,x\in\Rd\,,
\end{equation*}
and using the identity
\begin{equation*}
\Lg_\xi \varphis + \tfrac{1}{2} \babs{\upsigma\transp\grad\varphis}^2
\,=\, \Lg_\xi \varphis + \langle a y, \grad\varphis\rangle
+\tfrac{1}{2} \babs{\upsigma\transp(y-\grad\varphis)}^2
- \tfrac{1}{2} \abs{\upsigma\transp y}^2
\end{equation*}
to obtain (compare with \cref{E-eigen4})
\begin{equation}\label{PP4.1D}
\sA\varphis(x,\xi,y) 
+ \tfrac{1}{2} \babs{\upsigma\transp(x)\bigl(y-\grad\varphis(x)\bigr)}^2
+ \sR(x,\xi,y) \,\le\,\rhos\,.
\end{equation}

Using the function $\chi^{}_t$ in \cref{D4.1},
the identity
\begin{equation*}
\sA \chi^{}_t(\varphis) \,=\, \chi'_t(\varphis) \sA\varphis
+ \tfrac{1}{2} \chi''_t(\varphis)
\babs{\upsigma\transp\grad\varphis}^2\,,
\end{equation*}
and the definition of $\cH$,
we obtain from \cref{PP4.1D} that
\begin{equation}\label{PP4.1E}
\begin{aligned}
\sA (\chi^{}_t\comp\varphis)&(x,\xi,y) - \chi''_t\bigl(\varphis(x)\bigr)\,\cH(x)\\
& + \chi'_t\bigl(\varphis(x)\bigr)\Bigl(
\tfrac{1}{2} \babs{\upsigma\transp(x)\bigl(y-\grad\varphis(x)\bigr)}^2
+ \sR(x,\xi,y) - \rhos\Bigr) \,\le\,0\,.
\end{aligned}
\end{equation}
Let $\mu\in\eom_{\sA}\cap\cPs(\cZ)$, and without loss of generality assume
that $\mu\in\cPo(\cZ)$.
The integral of the first term in \cref{PP4.1E} with respect to $\mu$
vanishes by the definition of $\eom_{\sA}$.
Thus, we have
\begin{equation}\label{PP4.1F}
\begin{aligned}
\int_{\cZ}
\chi'_t\bigl(\varphis(x)\bigr)\Bigl(
\tfrac{1}{2} \babs{\upsigma\transp(x)\bigl(y-\grad\varphis(x)\bigr)}^2
&+ \sR(x,\xi,y) - \rhos\Bigr)\,\mu(\D{x},\D{\xi},\D{y})\\
&\,\le\, \int_{\Rd} \chi''_t\bigl(\varphis(x)\bigr)\,\cH(x)\,\eta(\D{x})\,,
\end{aligned}
\end{equation}
with $\eta(\cdot) = \int_{\Act\times\Rd}\mu(\cdot\,,\D{\xi},\D{y})$.
Since $\int\cH\D\eta<\infty$,
then taking limits as $t\to\infty$ in \cref{PP4.1F}, using
dominated convergence
together with the fact that $\chi''_t(s)\to 0$ as $t\to\infty$, we see that
the right-hand side of \cref{PP4.1F} goes to $0$. Also, using Fatou's lemma and
the fact that
$\chi'_t(s)\to 1$ as $t\to\infty$, we obtain from \cref{PP4.1F} that
\begin{equation}\label{PT4.1H}
\int_{\cZ}
\Bigl(\tfrac{1}{2} \babs{\upsigma\transp(x)\bigl(y-\grad\varphis(x)\bigr)}^2
+ \sR(x,\xi,y)\Bigr) \,\mu(\D{x},\D{\xi},\D{y})\,\le\,\rhos\,.
\end{equation}
This proves \cref{PP4.1C}.
Now, if we let
\begin{equation*}
\mu_*(\D{x},\D{\xi},\D{y})\df
\eta_{v_*}(\D{x}) \delta_{v_*(x)}(\D{\xi})\delta_{\nabla\varphi_*(x)}(\D{y})\,,
\end{equation*}
then
\begin{equation*}
\int_{\cZ} \sA f(z)\,\mu_*(\D{z}) \,=\, \int_{\Rd}
\widetilde\Lg_{v_*}^{\varphis} f(x)\,\eta_{v_*}(\D{x})\,=\,0
\quad\forall\,f\in\Cc^{2}_c(\Rd)\,,
\end{equation*}
which implies that $\mu_*\in\eom_{\sA}$.
Then, the second equality in \cref{PP4.1B} can be written as
\begin{equation}\label{PP4.1I}
\rhos \,=\, \int_{\cZ} \sR(z)\,\mu_*(\D{z})\,,
\end{equation}
while the first equality in \cref{PP4.1B} together with
the fact that $c$ is bounded above and $\rhos$ is finite
implies that $\mu_*\in\cPs(\cZ)$.
Therefore, $\mu_*\in\eom_{\sA}\cap\cPs(\cZ)$,
and the first equality in \cref{EP4.1A} now follows from \cref{PP4.1C,PP4.1I}.

We now turn to the proof of the second equality in \cref{EP4.1A}.
Note that it $\mu\notin\cPo(\cZ)$ then $F(0,\mu)=-\infty$.
On the other hand, if $\mu \notin\eom_\sA$ then, as also
stated in the proof of \cref{T2.2}, $\inf_{g \in \Cc^2_c(\Rd)}\,F(g,\mu)=-\infty$.
The remaining case is $\mu\in\eom_\sA\cap\cPs(\cZ)$, for which we have
$F(g,\mu)=\int_{\cZ} \sR(z)\,\mu(\D{z})$, thus proving the equality.

The second statement of the proposition follows directly from the arguments
used above.
\end{proof}

\begin{remark}
One can follow the argument in the proof of \cite[Theorem~1.4]{ABB-18},
using Radon--Nikodym derivatives
instead of densities, to show that every maximizing infinitesimal
ergodic occupation measure for \cref{EP4.1A} has the form
\begin{equation*}
\mu(\D{x},\D{\xi},\D{y}) \,=\,
\uppi(\D{x},\D{\xi})\,\delta_{\nabla\varphis(x)}(\D{y})\,,
\end{equation*}
where $\delta_y$ denotes the Dirac mass at $y\in\Rd$,
and $\uppi(\D{x},\D{\xi})$ is an optimal ergodic occupation measure of
the diffusion associated with operator $\sA^*$ defined by
\begin{equation*}
\sA^*\phi(x,\xi)\,\df\,\frac{1}{2}\trace\left(a(x)\nabla^{2}\phi(x)\right)
+ \bigl\langle b(x,\xi)+ a(x)\grad\varphis(x), \nabla \phi(x)\bigr\rangle
\end{equation*}
for $(x,\xi)\in\Rd\times\Act$ and
$f\in\Cc^2(\Rd)$.
We leave the verification of this assertion to the reader.
\end{remark}

We continue our analysis by investigating conditions on the model parameters which
imply that $\eom_{\sA}\cap\cPo(\cZ)\subset\cPs(\cZ)$.
We impose the following hypothesis on the matrix $a$.

\begin{assumption}\label{A4.1}
The matrix $a$ is bounded and
has a uniform modulus of continuity on $\Rd$, and is uniformly non-degenerate in
the sense that the minimum eigenvalue of $a$ is bounded away from zero on $\Rd$.
\end{assumption}

We start with the following lemma, which can be viewed as a generalization of
\cite[Lemma~3.3]{AB-18}.
\cref{A3.1}, which applies by default throughout the paper, need not  be
enforced in this lemma.

\begin{lemma}\label{L4.1}
Consider a linear operator in $\RR^d$, of the form
\begin{equation*}
\Lg \,\df\, \tfrac{1}{2} a^{ij}\partial_{ij} + b^i \partial_i + c\,,
\end{equation*}
and suppose that the matrix $a=\upsigma\upsigma\transp$ satisfies \cref{A4.1},
and the coefficients $b$ and $c$ are locally bounded and measurable.
Then, there exists a constant $\widetilde{C}_0$ such that any strong positive solution
$u \in\Sobl^{2,p}(\RR^d)$, $p>d$, to the equation
\begin{equation}\label{EL4.1B}
\Lg u(x) \,=\, 0 \quad \text{on } \RR^d
\end{equation}
satisfies
\begin{equation*}
\frac{\babs{\grad u(x)}}{u(x)} \,\le\, \widetilde{C}_0
\,\Bigl[1+ \sup_{y\in B_1(x)}\,\Bigl(\abs{b(y)} + \sqrt{\abs{c(y)}}\Bigr)\Bigr]
\qquad\forall\,x\in\Rd\,.
\end{equation*}
\end{lemma}

\begin{proof}
We use scaling.
For any fixed $x_0\in\RR^d$, with $\abs{x_0} \ge 1$, we define
\begin{equation*}
M_{x_0} \df 1+ \sup_{x\in B_3(x_0)}\,\Bigl(\abs{b(x)} + \sqrt{\abs{c(x)}}\Bigr)\,,
\end{equation*}
and the scaled function
\begin{equation*}
\Tilde{u}_{x_0}(y) \,\df\, u\bigl(x_0 + M_{x_0}^{-1} y \bigr)\,,\quad y\in\Rd\,,
\end{equation*}
and similarly for the functions $\Tilde{a}_{x_0}$, $\Tilde{b}_{x_0}$,
and $\Tilde{c}_{x_0}$.
The equation in \cref{EL4.1B} then takes the form
\begin{equation}\label{PL4.1A}
\frac{1}{2}\,\Tilde{a}^{ij}_{x_0}(y)\,
\partial_{ij}\Tilde{u}_{x_0}(y)
+ \frac{\Tilde{b}^i_{x_0}(y)}{M_{x_0}}\,
\partial_{i}\Tilde{u}_{x_0}(y)
+ \frac{\Tilde{c}_{x_0}(y)}{M_{x_0}^2}\,\Tilde{u}_{x_0}(y) \,=\, 0
\quad \text{on }\RR^d\,.
\end{equation}
It is clear from the hypotheses that the coefficients of \cref{PL4.1A}
are bounded in the ball $B_3$, with a bound independent of $x_0$, and that
the modulus of continuity and ellipticity constants of the matrix $\Tilde{a}_{x_0}$
in $B_3$ are independent of $x_0$.
We follow the argument in \cite[Lemma~3.3]{AB-18}, which is repeated here for
completeness.
First, by the Harnack inequality \cite[Theorem~9.1]{GilTru}, there exists
a positive constant $C_{\mathsf{H}}$ independent of the point $x_0$ chosen, such that
$\Tilde{u}_{x_0}(y)\le C_{\mathsf{H}}\, \Tilde{u}_{x_0}(y')$ for all $y,y' \in B_2$.
Let
\begin{equation*}
\Lg_0 \df \frac{1}{2}\,\Tilde{a}^{ij}_{x_0}(y)\,
\partial_{ij}
+ \frac{\Tilde{b}^i_{x_0}(y)}{M_{x_0}}\,\partial_i\,.
\end{equation*}
By a well known a priori estimate \cite[Lemma~5.3]{ChenWu}, there
exists a constant $C_{\mathsf{a}}$, again independent of $x_0$, such that,
\begin{equation}\label{PL4.1B}
\begin{aligned}
\bnorm{\Tilde{u}_{x_0}}_{\Sob^{2,p}(B_1)} &\,\le\, C_{\mathsf{a}}\,
\Bigl(\bnorm{\Tilde{u}_{x_0}}_{\Lp^{p}(B_2)}
+\bnorm{\Lg_0\,\Tilde{u}_{x_0}}_{\Lp^{p}(B_2)}\Bigr)\\
&\,\le\, C_{\mathsf{a}}\,\biggl(1+\sup_{y\in B_2}\,
\frac{\Tilde{c}_{x_0}(y)}{M_{x_0}^2}\biggr)\,
\bnorm{\Tilde{u}_{x_0}}_{\Lp^{p}(B_2)}\\
&\,\le\,\widetilde{C}_1\,\Tilde{u}_{x_0}(0)\,,
\end{aligned}
\end{equation}
where in the last inequality, we used the Harnack property.
Clearly then, the resulting constant $\widetilde{C}_1$ does not depend on $x_0$.
Next, invoking Sobolev's theorem, which asserts the compactness of the embedding
$\Sob^{2,p}\bigl(B_1(x_0)\bigr)\hookrightarrow \Cc^{1,r}\bigl(B_1(x_0)\bigr)$,
for $p>d$
and $r<1-\frac{d}{p}$ (see \cite[Proposition~1.6]{ChenWu}), and
combining this with \cref{PL4.1B}, we obtain
\begin{equation*}
\sup_{y\in B_1}\,\babs{\grad\Tilde{u}_{x_0}(y)} \,\le\,
\widetilde{C}_2\, \Tilde{u}_{x_0}(x_0)
\end{equation*}
for some constant $\widetilde{C}_2$ independent of $x_0$.
Thus
\begin{equation}\label{PL4.1C}
\frac{\abs{\grad\Tilde{u}_{x_0}(0)}}{\Tilde{u}_{x_0}(0)}
\,\le\, \widetilde{C}_2
\qquad\forall\,x_0\in B_1^c\,.
\end{equation}
Using \cref{PL4.1C} and the identity
$\grad{u}(x_0) = M_{x_0}
\,\grad\Tilde{u}_{x_0}(0)$ for all $x_0\in B_1^c$,
we obtain
\begin{equation*}
\frac{\babs{\grad{u}(x_0)}}{{u}(x_0)} \,=\,
M_{x_0}\,
\frac{\babs{\grad\Tilde{u}_{x_0}(0)}}{\Tilde{u}_{x_0}(0)}
\,\le\,\widetilde{C}_2\,
\biggl[1+ \sup_{x\in B_3(x_0)}\,\Bigl(\abs{b(x)} + \sqrt{\abs{c(x)}}\Bigr)\biggr]
\qquad \forall\,x_0\in B_1^c\,.
\end{equation*}
Of course $B_3(x_0)$ is arbitrary. The same is true with any radius, with
perhaps a different constant.
This completes the proof.
\end{proof}

\begin{remark}
\Cref{L4.1} should be compared with similar gradient estimates in the literature.
Its benefit is that it matches or exceeds the estimates
in \cite[Lemma~5.1]{Metafune-05}
and \cite[Theorem~A.2]{Chasseigne-19}, without requiring
any regularity on the coefficients.
\end{remark}

\begin{assumption}\label{A4.2}
One of the following holds:
\begin{enumerate}
\item[\ttup a] The function $-\rc$ is inf-compact.

\item[\ttup b] The drift $b$ satisfies
\begin{equation}\label{E-weak}
\max_{(x,\xi)\in B_r^c\times\Act}\;
\frac{ \bigl\langle b(x,\xi),\, x\bigr\rangle^{-}}{\abs{x}^{2}}
\;\xrightarrow[r\to\infty]{}\;0\,.
\end{equation}

\item[\ttup c]
There exists a constant $\widehat{C}_0$ such that
(compare this with \cite[Theorem~3.1\,(b)]{AB-18b})
\begin{equation*}
\frac{\cH(x)}{\bigl(1+\abs{\varphis(x)}\bigr)\,\bigl(1+\abs{\rc(x,\xi)}\bigr)}
\,\le\,\widehat{C}_0\qquad\forall\,(x,\xi)\in\Rd\times\Act\,,
\end{equation*}
where $\varphis=\log\Phis$, and $\Phis$ is as in \cref{T3.1}.
\end{enumerate}
\end{assumption}

\begin{remark}
\Cref{A4.2}\,(c) is not specified in terms of the parameters of
the equation. However, \cref{A4.1} together with the hypothesis that
$\frac{\abs{b}^2}{1+\abs{\rc}}$ is bounded implies \cref{A4.2}\,(c).
This is asserted by \cref{L4.1}.
See also \cref{L4.3} later in this section.
\end{remark}

We have the following estimate concerning the growth of the function $\Phis$
in \cref{T3.1}. This does not require the uniform ellipticity hypothesis
in \cref{A4.1}.

\begin{lemma}\label{L4.2}
Grant \cref{A4.2} part \textup{(a)} or \textup{(b)}.
Then there exists a function $\zeta\colon(0,\infty)\to(0,\infty)$,
with $\lim_{r\to\infty}\zeta(r)=\infty$, such that the solution $\Phis$ in
\cref{ET3.1A} satisfies
\begin{equation}\label{EL4.2A}
\babs{\log\Phis(x)} \,\ge\, \zeta(r)\,\log\bigl(1+\abs{x}\bigr)
\qquad\forall\,x\in B_r^c\,.
\end{equation}
\end{lemma}

\begin{proof}
We start with part (a).
Let $\alpha\colon(0,\infty)\to(0,\infty)$ be a strictly increasing function,
satisfying $\alpha(r)\to\infty$ and $\frac{\alpha(r)}{r}\to0$ as $r\to\infty$, and
\begin{equation}\label{PL4.2C}
\log\alpha(r) \,\ge\, \log r - \inf_{B_r^c}\,\abs{\varphis}^{\nicefrac{1}{3}}\,.
\end{equation}
This is always possible.
A specific function satisfying these properties is given by
\begin{equation*}
\alpha(r) \,\df\, \sqrt r + \sup_{s\in (0, r]}\,
\biggl(s \exp\Bigl(-\inf_{B_r^c}\,\abs{\varphis}^{\nicefrac{1}{3}}\Bigr)\biggr)\,.
\end{equation*}
Let $c_1$ be a constant such that $\babs{\Lg_{v_*} (\log\abs{x})} \le c_1$
for all $\abs{x}>1$. Such a constant exists since $\upsigma$ and
$b$ have at most linear growth in $\abs{x}$ by \cref{E-growth}.
We define
\begin{equation}\label{PL4.2D}
\kappa(r) \,\df\, \min\, \biggl(\sqrt r\,,\,\frac{1}{c_1}\inf_{B_r^c\times\Act}\,
\babs{\rc(x,\xi)-\rhos}^{\nicefrac{1}{2}}\,,\,
\inf_{B_r^c}\,\abs{\varphis}^{\nicefrac{1}{3}}\biggr)\,.
\end{equation}
Since the functions $-\varphis$ and $-\rc$ are inf-compact, it
is clear that $\kappa(r)\to\infty$ as $r\to\infty$.

Define the family of functions
\begin{equation*}
h_r(x)\,\df\,-\kappa(r)\bigl(\log\abs{x} - \log \alpha(r)\bigr)\,,
\qquad r\ge1\,,\ x\in B_r^c\,.
\end{equation*}
Note that for any $g\in\Cc^2(\Rd)$ we have
\begin{equation}\label{PL4.2E}
\Lg_\xi \chi^{}_t(g) \,=\, \chi'_t(g) \Lg_\xi(g) +
\frac{1}{2}\chi''_t(g) \babs{\upsigma\transp\nabla g}^2\,.
\end{equation}
Thus, applying \cref{PL4.2E} and the bound $\babs{\Lg_{v_*} (\log\abs{x})} \le c_1$,
we obtain
\begin{equation}\label{PL4.2F}
\begin{aligned}
\widetilde\Lg^{\varphis}_{v_*}
\chi_t\bigl(h_r(x)\bigr) \,\le\, c_1\,\kappa(r)\,\chi'_t & \bigl(h_r(x)\bigr)
+ \bigl\langle a(x)\nabla\varphis(x), \nabla \chi_t\bigl(h_r(x)\bigr) \bigr\rangle\\
&+ \frac{1}{2}\, \chi_t''\bigl(h_r(x)\bigr)\,\babs{\upsigma\transp(x)\nabla h_r(x)}^2
\qquad\forall\,x\in B_r^c\,.
\end{aligned}
\end{equation}
Combining \cref{PP4.1A,PL4.2F}, and completing the squares, we have
\begin{equation}\label{PL4.2G}
\begin{aligned}
\widetilde\Lg^{\varphis}_{v_*}\bigl(\chi_t\comp h_r-\varphis\bigr)(x)
&\,\le\, \rc_v(x)-\rhos + c_1\,\kappa(r)\,\chi'_t\bigl(h_r(x)\bigr)\\
&\mspace{5mu}
+ \frac{1}{2}\, \chi_t''\bigl(h_r(x)\bigr)\,\babs{\upsigma\transp(x)\nabla h_r(x)}^2
+ \frac{1}{2}\babs{\upsigma\transp(x)\nabla \chi_t\bigl(h_r(x)\bigr)}^2\\
&\mspace{50mu} -\frac{1}{2}\babs{\upsigma\transp(x)
\bigl[\nabla\varphis(x) - \nabla \chi_t\bigl(h_r(x)\bigr)\bigr]}^2\,.
\end{aligned}
\end{equation}
Recall that $\chi'\le1$, and $\chi''\le1$.
Choose $r$ large enough so that $\varphis<-1$ on $B_r^c$.
It then follows by the definitions in \cref{PL4.2C,PL4.2D} that
$\varphis- \chi_t\comp h_r<0$ on $\partial B_r$ for all $t\ge0$.
Also, for each $t>0$, the difference $\varphis- \chi_t\comp h_r$ is negative
outside some compact set by the inf-compactness of $-\varphis$.
Note also that $\abs{\nabla h_r}\le \frac{\kappa(r)}{r}$ on $B_r^c$.
Hence \cref{E-growth,PL4.2D} imply that there exists $r_0$ such
the right-hand side of \cref{PL4.2G} is negative on $B_r^c$ for all
$r>r_0$ and all $t\ge0$.
An application of the strong maximum principle then shows
that $\varphis<h_r$ on $B_r^c$ for all $r>r_0$.

Now, note that
\begin{equation*}
\log \frac{\abs{x}}{\alpha(r)} \ge \frac{1}{2} \log\bigl(1+\abs{x}\bigr)
\qquad\text{when\ \ } \abs{x}\ge \max\,\bigl(1,2\bigl(\alpha(r)\bigr)^2\bigr)\,.
\end{equation*}
Since $\alpha(r)$ is strictly increasing, the inequality
\cref{EL4.2A} holds with
\begin{equation*}
\zeta(r) \df \frac{1}{2}\,\kappa\Bigl(\alpha^{-1}
\bigr(\sqrt{\nicefrac{r}{2}}\bigr)\Bigr)\qquad\text{for all\ }
r\ge 2\bigl(\alpha(r_0)\bigr)^2\,.
\end{equation*}
This completes the proof under \cref{A4.2}\,(a)\,.

The proof under part (b) of the assumption is similar.
The only difference is that here we use the fact that
$m_r\,\df\,\sup_{x\in B_r^c}\,\bigl(\Lg_{v_*} (\log\abs{x})\bigr)^-\to0$
as $t\to\infty$,
which is implied by \cref{E-weak}.
Thus with $\epsilon>0$ any constant such that $\rhos-\rc>\epsilon$ outside some
compact set, we choose $\kappa(r)$ as
\begin{equation*}
\kappa(r) \,\df\, \min\, \biggl(\sqrt r\,,\,\sup_{B_r^c\times\Act}\,
\frac{\epsilon}{2 \sqrt m_r}\,,\,
\inf_{B_r^c}\,\abs{\varphis}^{\nicefrac{1}{3}}\biggr)\,.
\end{equation*}
The rest is completely analogous to the analysis above.
This concludes the proof.
\end{proof}

The first part of the theorem which follows is quite technical,
but identifies a rather deep property of the ergodic occupation measures
of the operator $\sA$. It shows that under \cref{A4.1,A4.2}\,\textup{(a)} or \textup{(b)},
or \cref{A4.2}\,\textup{(c)}, if such a measure $\mu$ is feasible for the
maximization problem, or in other words, it satisfies
$\int_{\cZ} \sR(z)\,\mu(\D{z}) > -\infty$, then it necessarily has
``finite average'' entropy, that is $\int \cH\,\D\mu<\infty$, or equivalently,
it belongs in the class $\cPs(\cZ)$.
The proof uses the method of contradiction.
We first show that if such a measure $\mu$ is not in the class $\cPs(\cZ)$,
then the left hand side of  \cref{PP4.1F} grows at a geometric rate
as a function of $t$. Then we obtain a contradiction by
evaluating the right-hand side of \cref{PP4.1F}
using this geometric growth together
with the bound in \cref{L4.2}.

\begin{theorem}\label{T4.1}
\begin{enumerate}
\item[\ttup i]
Under \cref{A4.1,A4.2}\,\textup{(a)} or \textup{(b)},
or \cref{A4.2}\,\textup{(c)}, we have $\eom_{\sA}\cap\cPo(\cZ)\subset\cPs(\cZ)$.
This of course implies by \cref{P4.1} that
\begin{equation*}
\rhos \,=\, \max_{\mu\in\eom_{\sA}}\,\int_{\cZ} \sR(z)\,\mu(\D{z})
\,=\, \adjustlimits\sup_{\mu\in\cP(\cZ)} \inf_{g \in \Cc^2_c(\Rd)}\,F(g,\mu) \,.
\end{equation*}

\item[\ttup{ii}]
Let \cref{A4.1} hold, and suppose that
\begin{equation}\label{ET4.1C}
\sup_{x\in \Rd}\, \frac{\cH(x)}{1+\abs{\varphis(x)}} \,<\,\infty\,.
\end{equation}
Then
\begin{equation}\label{ET4.1D}
\rhos \,=\, \adjustlimits\inf_{g \in \Cc^2_c(\Rd)\,} \sup_{\mu\in\cP(\cZ)}\,F(g,\mu)\,.
\end{equation}
\end{enumerate}
\end{theorem}

\begin{proof}
We first prove part (i) under under \cref{A4.2}\,(a) or (b).
We argue by contradiction. Let $\mu\in\eom_{\sA}\cap\cPo(\cZ)$,
and suppose that $\mu\notin\cPs(\cZ)$.
As in the proof of \cref{P4.1} we let
$\eta(\cdot) = \int_{\Act\times\Rd}\mu(\cdot\,,\D{\xi},\D{y})$.
Let $\mathscr{I}_1(t)$ and $\mathscr{I}_2(t)$ denote the left and the right-hand
side of \cref{PP4.1F}, respectively, and define
\begin{equation*}
\mathcal{I}(t) \,\df\, \int_{\Rd}
\chi'_t\bigl(\varphis(x)\bigr)\,\cH(x)\,\eta(\D{x})\,.
\end{equation*}
Then of course $\mathcal{I}(t)\to\infty$ as $t\to\infty$ by the hypothesis.
Expanding $\mathscr{I}_1(t)$ we see that
\begin{equation*}
\mathscr{I}_1(t) \,=\, \mathcal{I}(t) - \int_{\cZ}\chi'_t\bigl(\varphi^*(x)\bigr)
\bigl\langle a(x) y,\grad\varphi^*(x)\bigr\rangle\, \D{\mu}
+ \int_{\cZ} \chi'_t(\varphi^*(x)) (\rc-\rhos)\,\D{\mu}\,.
\end{equation*}
Since $\int \sR\,\D\mu$ is finite, it follows that
$\int_{\cZ} \abs{\upsigma\transp y}^2 \D{\mu}$ and
$\int_{\cZ} \max\{-\rc,0\} \, \D{\mu}$ are also finite.
Moreover, the second assertion and the fact that $c$ is bounded above
imply that
$\int_{\cZ} |\rc| \, \D{\mu}<\infty$.
Thus, using the Cauchy--Schwarz inequality in the above display
and the fact $|\chi'_t|$ is bounded, we have
\begin{equation}\label{PT4.1I}
\alpha_0(t) - \alpha_1(t)\sqrt{\mathcal{I}(t)} + \mathcal{I}(t)
\,\le\,\mathscr{I}_1(t)
\,\le\,\alpha_0(t) + \alpha_1(t)\sqrt{\mathcal{I}(t)} + \mathcal{I}(t)
\end{equation}
for some constants $\alpha_0(t)$ and $\alpha_1(t)$ which are bounded
in $t\in[0,\infty)$.

First suppose that over some sequence $t_n\to\infty$ we have
$\frac{\mathscr{I}_2(t_n)}{\mathscr{I}_1(t_n)}\to\delta<1$ as $n\to\infty$.
This implies by \cref{PT4.1I} that
$\frac{\mathscr{I}_2(t_n)}{\mathcal{I}(t_n)}\to\delta$.
However, if this is the case, then the inequality
\begin{equation*}
\alpha_0(t_n) - \alpha_1(t_n)\sqrt{\mathcal{I}(t_n)} +
\Bigl(1- \tfrac{\mathscr{I}_2(t_n)}{\mathcal{I}(t_n)}\Bigr)\mathcal{I}(t_n)\,\le\,0\,,
\end{equation*}
which is implied by \cref{PP4.1F,PT4.1I},
contradicts the fact that $\mathcal{I}(t)\to\infty$ as $t\to\infty$.
Thus we must have $\liminf_{t\to\infty} \frac{\mathscr{I}_2(t)}{\mathscr{I}_1(t)}\ge1$,
and same applies to the fraction $\frac{\mathscr{I}_2(t)}{\mathcal{I}(t)}$.

Define
\begin{equation*}
g_k \,\df\, \int_{\Rd}\cH(x)\,
\Ind_{\{ -2k<\varphis(x)< -2k+2\}}\, \eta(\D{x})\,,\qquad
k\in\NN\,.
\end{equation*}
We have $\mathcal{I}(2n)\ge \sum_{k=1}^n g_k$ for $n\in\NN$,
by definition of these quantities.
Recall that $\mathscr{I}_2(t)$ is defined as the right-hand
side of \cref{PP4.1F}.
Note then that, since $\chi''<1$, we have
$\mathscr{I}_2(2n)< \delta g_{n+1}$ for some $\delta<1$.
Therefore, since $\liminf_{t\to\infty}\,\frac{\mathscr{I}_2(t)}{\mathcal{I}(t)}\ge1$,
there exists $n_0\in\NN$ such that
\begin{equation}\label{AB1}
S_n \,\df\,\sum_{k=1}^{n} g_k \,\le\, g_{n+1}\qquad\forall\,n\ge n_0\,.
\end{equation}
Thus $S_{n+1} - S_n = g_{n+1} \ge S_n$, which implies that $S_{n+1}\ge 2 S_n$.
This of course means that $S_n$ diverges at a geometric rate in $n$,
that is, $S_{n}\ge 2^{n-1} S_1$.
Let $h$ denote the inverse of the map $y\mapsto\zeta(y)\log(1+y)$.
Note that $\cH(x)\le C(1+\abs{x}^p)$ for
some positive constants $C$ and $p$ by \cref{L4.1}
and the hypothesis that $\rc$ has polynomial growth in \cref{A3.1}\,(ii).
Thus, by \cref{L4.2}, we obtain
\begin{align*}
g_n &\,\le\, C \int_{\Rd} (1+|x|^p)\,\Ind_{\{-2n<\varphis(x)< -2n+2\}}\,\eta(\D{x})\\
&\,\le\, C \int_{\Rd} (1+|x|^p)\,\Ind_{\{\zeta(|x|)\log(1+|x|)<2n\}}\, \eta(\D{x})\\
&\,\le\, C \bigl(1 + h(2n)^p\bigr)
\end{align*}
for all $n\in\NN$.
However, this implies from \cref{AB1} that
\begin{equation*}
\begin{aligned}
\log 2 \,\le\, \limsup_{n\to\infty}\,\frac{\log S_n}{n}
&\,\le\, C'\limsup_{n\to\infty}\,\frac{\log h(n)}{n}\\
&\,=\, C' \limsup_{k\to\infty}\,\frac{\log k}{ \zeta(k) \log(1+k)}\,=\,0
\end{aligned}
\end{equation*}
for some constant $C'$, and we reach a contradiction.
Therefore, $\eom_{\sA}\cap\cPo(\cZ)\subset\cPs(\cZ)$.

Moving on to the proof under \cref{A4.2}\,(c),
we replace the function $\chi^{}_t$ in \cref{D4.1} by
a function $\Tilde{\chi}^{}_t$ defined as follows.
For $t>0$, we let $\Tilde{\chi}^{}_t$ be a convex $\Cc^2(\RR)$ function such that
$\Tilde{\chi}^{}_t(s)= s$ for $s\ge -t$, and $\Tilde{\chi}^{}_t(s) =
\text{constant}$ for $s\le -t\E^2$.
Then $\Tilde{\chi}'_t$ and $\Tilde{\chi}''_t$ are nonnegative.
In addition, we select $\Tilde{\chi}^{}_t$ so that
$\Tilde{\chi}''_t(s) \le -\frac{1}{s}$ for
$s\in[-t\E^2,-t]$  and $t\ge0$.
This is always possible.
We follow the same analysis as in the proof of \cref{P4.1}, with the function
$\Tilde{\chi}^{}_t$ as chosen, and obtain
\begin{equation}\label{PT4.1J}
\begin{aligned}
\int_{\cZ}
\Tilde{\chi}'_t&\bigl(\varphis(x)\bigr)\Bigl(
\tfrac{1}{2} \babs{\upsigma\transp(x)\bigl(y-\grad\varphis(x)\bigr)}^2
+ \sR(x,\xi,y) - \rhos\Bigr)\,\mu(\D{x},\D{\xi},\D{y})\\
&\,\le\,\int_{\Rd} \Tilde{\chi}''_t\bigl(\varphis(x)\bigr)\,\cH(x)\,\eta(\D{x})\\
&\,\le\, \int_{\Rd}\frac{\cH(x)}{\abs{\varphis(x)}}\,\Ind_{A_t}(x) \,\eta(\D{x})\\
&\,\le\, \widehat{C}_0\,
\int_{\Rd\times\Act\times\Rd} \frac{1+\abs{\varphis(x)}}{\abs{\varphis(x)}}
\bigl(1+\abs{\rc(x,\xi)}\bigr)
\Ind_{A_t}(x) \,\mu(\D{x},\D{\xi},\D{y})\,,
\end{aligned}
\end{equation}
where $A_t \df \{x\colon \varphis(x)\le -t\}$.
The integral on the right-hand side of \cref{PT4.1J} vanishes as $t\to\infty$ by
the hypothesis that $\int \rc\,\D\mu>-\infty$,
so again we obtain \cref{PT4.1H} which implies the result.
This completes the proof of part (i).

We continue with part (ii).
We use a $\Cc^2$ convex function $\Hat\chi_t\colon\RR\to\RR$, for $t\ge1$,
satisfying $\Hat\chi_t(s)=s$ for $s\le -t$,
$\Hat\chi''_t(s)\le -\frac{1}{s\log \abs{s}}$
for $s<-t$, and $\Hat\chi_t(s)=\text{constant}$ for $s\ge \Hat\zeta(t)$,
for some $\Hat\zeta(t)<-t$.
We let $h_t(x) = \Hat\chi_t\bigl(\varphis(x)\bigr)$.
We may translate $\varphis$ so that it is smaller than $-1$ on $\Rd$.
By \eqref{PP4.1E}, we have
\begin{equation}\label{PT4.1L}
\begin{aligned}
\sA h_t(z) +\sR(z)-\rhos &\,\le\,
\bigl[1-\Hat\chi'_t\bigl(\varphis(x)\bigr)\bigr]
\bigl(\sR(z)-\rhos\bigr)\\
&\mspace{10mu}-\tfrac{1}{2} \Hat\chi'_t\bigl(\varphis(x)\bigr)
 \babs{\upsigma\transp(x)\bigl(y-\grad\varphis(x)\bigr)}^2
+ \Hat\chi''_t\bigl(\varphis(x)\bigr)\,\cH(x)\,.
\end{aligned}
\end{equation}
We claim that given any $\epsilon>0$ there exists $t>0$ such that
$F(h_t,\mu)\le \rhos+\epsilon$ for all $\mu\in\cP(\cZ)$.
This of course suffices to establish \cref{ET4.1D}.

By \cref{A3.1}\,(iii) there exists $t_1>0$ such that the first term on the
right-hand side of \cref{PT4.1L} is nonpositive for all $t\ge t_1$.
Also, using the definition of $\Hat\chi$, we have
\begin{equation*}
\Hat\chi''_t\bigl(\varphis(x)\bigr)\,\cH(x) \,\le\,
\frac{\cH(x)}{\abs{\varphis(x)} \log \abs{\varphis(x)}}\,
\Ind\{x\in\Rd\colon \varphis(x) \le -t\}\,\xrightarrow[t\to\infty]{}\,0
\end{equation*}
by the hypothesis, and
since $-\varphis$ is inf-compact by \cref{T3.1}.
This proves the claim, and completes the proof.
\end{proof}

There is a large class of problems which satisfy \cref{ET4.1C}.
It consists of equations with $\abs{b}^2+\abs{\rc}$ having at most linear growth
in $\abs{x}$ and $\abs{x}^{-1}\langle b,x\rangle^-$ growing no faster than
$\abs{\rc}^2$.
This fact is stated in the following lemma.

\begin{lemma}\label{L4.3}
Grant \cref{A4.1} and suppose that
\begin{equation*}
\sup_{(x,\xi)\in\Rd\times\Act}\, \max\;
\biggl(\frac{\langle b(x,\xi),x\rangle^-}{1+\abs{x}\abs{\rc(x,\xi)}},
\frac{\abs{b(x,\xi)}^2+\abs{\rc(x,\xi)}}{1+\abs{x}}\biggr)\,<\,\infty\,.
\end{equation*}
Then \cref{ET4.1C} holds.
\end{lemma}

\begin{proof}
We use the function $\chi_t$ in \cref{D4.1}.
Let $\Tilde{r}>0$ be such that $\rhos-\rc(x,\xi)>\delta>0$
on $B_{\Tilde{r}}^c\times\Act$.
Note that there exists a constant $C$ such that
\begin{equation*}
\widetilde\Lg^{\varphis}_{v_*} \chi_t \bigl(\epsilon (\Tilde{r}-\abs{x})\bigr)
\,\le\, C\epsilon\bigl(1 + \abs{x}^{-1}\langle b_{v_*}(x),x\rangle^-
+ \abs{\nabla\varphis(x)}\bigr)
\quad \forall\, t>0\,.
\end{equation*}
Thus for some $\epsilon>0$ small enough, using \cref{PP4.1A}, we obtain
\begin{equation*}
\widetilde\Lg^{\varphis}_{v_*}\bigl(\varphis(x)
-\chi_t \bigl(\epsilon (\Tilde{r}-\abs{x})\bigr)\bigr) \,>\,0
\qquad\forall\, x\in B_{\Tilde{r}}^c\,,\ \ \forall\, t>0\,.
\end{equation*}
An application of the strong maximum principle then shows that
$\varphis(x)\le \epsilon (\Tilde{r}-\abs{x})^-$.
Therefore, using \cref{L4.1}, we obtain
\begin{equation*}
\babs{\grad\varphis(x)}^2 \,\le\,
C' (1+\abs{x}) \,\le\, C'\bigl(1+\Tilde{r}-\epsilon^{-1}\varphis(x)\bigr)
\qquad \forall\, x\in B_{\Tilde{r}}^c\,,
\end{equation*}
for some constant $C'$.
\end{proof}

We next present the variational formula over functions in
$\Cc^2(\Rd)$ whose derivatives up to second order have at most polynomial growth
in $\abs{x}$.
Let $\Cc_{\mathsf{pol}}^2(\Rd)$ denote this space of functions.

\begin{theorem}\label{T4.2}
Under \cref{A3.1} alone, we have
\begin{equation}\label{ET4.2A}
\rhos \,= \, \adjustlimits\inf_{g \in \Cc^2(\Rd)}
\sup_{\mu\in\cP(\cZ)}\,F(g,\mu)\,.
\end{equation}
Under \cref{A4.1,A4.2}\,\textup{(a)} or \textup{(b)}, we have
\begin{equation}\label{ET4.2B}
\rhos \,= \, \adjustlimits\inf_{g \in \Cc_{\mathsf{pol}}^2(\Rd)}
\sup_{\mu\in\cP(\cZ)}\,F(g,\mu)
\,=\, \adjustlimits\sup_{\mu\in\cP(\cZ)}
\inf_{g \in \Cc_{\mathsf{pol}}^2(\Rd)}\, F(g,\mu)\,.
\end{equation}
\end{theorem}

\begin{proof}
By \cref{E-eigen4,ET3.1A} we have
\begin{equation*}
\adjustlimits\max_{\xi\in\Act\,}\max_{y\in\Rd}\;
\bigl[\sA \varphis (x,\xi,y) + \sR(x,\xi,y)\bigr]\,=\, \rhos\,.
\end{equation*}
Since $\varphis\in\Cc^2(\Rd)$, this implies that
\begin{equation*}
\adjustlimits\inf_{g \in \Cc^2(\Rd)}
\sup_{\mu\in\cP(\cZ)}\,F(g,\mu) \,\le\, \rhos\,.
\end{equation*}
On the other hand, by \cref{T3.1}\,(d), it follows that for any 
$g\in \Cc^2(\Rd)$ we have
\begin{equation*}
\sup_{z\in\cZ}\, \bigl[\sA g(z) + \sR(z)\bigr] \,\ge\, \rhos\,,
\end{equation*}
which then implies the converse inequality
\begin{equation*}
\adjustlimits\inf_{g \in \Cc^2(\Rd)}
\sup_{\mu\in\cP(\cZ)}\,F(g,\mu) \,\ge\, \rhos\,.
\end{equation*}
This proves \cref{ET4.2A}.

Concerning \cref{ET4.2B}, the first equality follows as in the preceding
paragraph since $\varphis\in\Cc_{\mathsf{pol}}^2(\Rd)$ by
Assumptions \ref{A3.1}\,(i)--(ii) and \ref{A4.1}, and \cref{L4.1}.
Turning now our attention to the second equality in \cref{ET4.2B},
recall from the proof of \cref{P4.1}
that $\eta_{v_*}$ denotes the invariant probability measure of
$\widetilde\Lg_{v_*}^{\varphis}$.
Under \cref{A4.2}\,(a) or (b), \cref{L4.2} shows
that $\Phis^{-1}(x)$ grows faster in $\abs{x}$ than any polynomial.
Therefore, $\int_\Rd \abs{x}^n\, \eta_{v_*}(\D{x})<\infty$
for all $n\in\NN$ by \cref{ET3.1B}.
Since $\abs{\nabla\varphis(x)}$ has at most polynomial growth,
and $b$ has at most linear growth, we obtain
\begin{equation}\label{PT4.2A}
\int_{\Rd}\babs{\widetilde\Lg_{v_*}^{\varphis} f(x)}\,\eta_{v_*}(\D{x})\,<\,\infty
\qquad\forall\,f\in\Cc_{\mathsf{pol}}^2(\Rd)\,.
\end{equation}
Continuing, if \cref{PT4.2A} holds, then
it is standard to show by employing a cut-off function, that
\begin{equation}\label{PT4.2B}
\int_{\Rd}\widetilde\Lg_{v_*}^{\varphis} f(x)\,\eta_{v_*}(\D{x})\,=\,0
\qquad\forall\, f\in\Cc_{\mathsf{pol}}^2(\Rd)\,.
\end{equation}
Let $\mu_*\in\eom_\sA$ denote the ergodic occupation measure corresponding
to $\eta_{v_*}$, that is,
\begin{equation*}
\mu_*(\D{x},\D{\xi},\D{y})\,=\,\eta_{v_*}(\D{x})\,\delta_{v_*(x)}(\D{\xi})\,
\delta_{\nabla\varphis}(\D{y})\,.
\end{equation*}
\Cref{PT4.2B} implies that
\begin{equation}\label{PT4.2C}
F(g,\mu_*) \,=\, \,\int_{\cZ} \sR(z)\,\mu_*(\D{z})
\,=\, \rhos\qquad\forall\,g \in \Cc_{\mathsf{pol}}^2(\Rd)\,.
\end{equation}
Since
\begin{equation*}
\adjustlimits\sup_{\mu\in\cP(\cZ)}\inf_{g \in \Cc_{\mathsf{pol}}^2(\Rd)}\;
F(g,\mu)\,\le\,
\adjustlimits\inf_{g \in \Cc_{\mathsf{pol}}^2(\Rd)} \sup_{\mu\in\cP(\cZ)}\,F(g,\mu)\,,
\end{equation*}
the second equality in \cref{ET4.2B} then follows by \cref{ET4.2A,PT4.2C}.
\end{proof}

\section{The risk-sensitive cost minimization problem}\label{S5}

Using \cref{L4.1}, we can improve the main result in
\cite{AB-18} which assumes bounded drift and running cost.

We say that a function $f\colon\cX\to\RR$ defined on a locally compact space
is \emph{coercive, or near-monotone,
relative to a constant $\beta\in\RR$} if there exists a compact set $K$ such
that $\inf_{K^c}\,f >\beta$.
Recall that an admissible control $\xi$ for \cref{E-sde1} is a process
$\xi_t(\omega)$ which takes values in $\Act$, is jointly measurable in
$(t,\omega)\in[0,\infty)\times\Omega$, and is
non-anticipative, that is,
for $s < t$, $W_{t} - W_{s}$ is independent of $\sF_{s}$ given in \cref{E-sF}.
We let $\Uadm$ denote the class of admissible controls, and $\Exp^x_\xi$ the
 expectation operator on the canonical space of the
process under the control $\xi\in\Uadm$, conditioned on the
process $X$ starting from $x\in\RR^{d}$ at $t=0$.

Let $c\colon\Rd\times\Act\to\RR$ be continuous, and
Lipschitz continuous in its first argument
uniformly with respect to the second.
We define the risk-sensitive penalty by
\begin{equation*}
\sE^x_\xi\,=\, \sE^x_\xi(c)\,\df\, \limsup_{T\to\infty}\;\frac{1}{T}\,
\log\Exp^x_\xi\Bigl[\E^{\int_{0}^{T} c(X_{t},\xi_t)\,\D{t}}\Bigr]\,,
\quad \xi\in\Uadm\,,
\end{equation*}
and the risk-sensitive optimal values by
$\sE^x_* \df \inf_{\xi\in\,\Uadm}\,\sE^x_\xi$, and
$\sE_* \df \inf_{x\in\,\Rd}\,\sE^x_*$.
Let
\begin{equation*}
\widehat\cG f(x) \,\df\, \frac{1}{2}\trace\left(a(x)\nabla^{2}f(x)\right)
+ \min_{\xi\in\Act}\, \bigl[\bigl\langle b(x,\xi),
\grad f(x)\bigr\rangle + c(x,\xi) f(x)\bigr]\,,\quad f\in\Cc^2(\Rd)\,,
\end{equation*}
and
\begin{equation*}
\widehat\lambda_*\,=\,\widehat\lambda_*(c)\,\df\,\inf\,\Bigl\{\lambda\in\RR\,
\colon \exists\, \varphi\in\Sobl^{2,d}(\Rd),\ \varphi>0, \
\widehat\cG\varphi -\lambda\varphi\le 0 \text{\ a.e.\ in\ } \Rd\Bigr\}\,.
\end{equation*}
We say that $\widehat\lambda_*$ is \emph{strictly monotone at $c$ on the right}
if $\widehat\lambda_*(c+h)>\widehat\lambda_*(c)$ for all non-trivial nonnegative
functions $h$ with compact support.

\Cref{E-Prop} below improves \cite[Proposition~1.1]{AB-18}.
We first state the assumptions.

\begin{assumption}\label{A5.1}
In addition to \cref{A4.1} we require the following.
\begin{enumerate}
\item[\ttup i]
The drift $b$ and running cost $c$ satisfy, for some $\theta\in[0,1)$
and a constant $\kappa_0$, the bound
\begin{equation*}
\abs{b(x,\xi)} \,\le\, \kappa_0\bigl(1+\abs{x}^\theta\bigr)\,,\quad\text{and\ \ }
\abs{c(x,\xi)} \,\le\, \kappa_0\bigl(1+\abs{x}^{2\theta}\bigr)
\end{equation*}
for all $(x,\xi)\in\Rd\times\Act$.

\item[\ttup{ii}]
The drift $b$ satisfies
\begin{equation}\label{EA5.1A}
\frac{1}{\abs{x}^{1-\theta}}\;
\max_{\xi\in\Act}\;\bigl\langle b(x,\xi),\, x\bigr\rangle^{+}
\;\xrightarrow[\abs{x}\to\infty]{}\;0\,.
\end{equation}
\end{enumerate}
\end{assumption}

\begin{proposition}\label{E-Prop}
Grant \cref{A5.1}, and suppose that $c$ is coercive relative to $\sE_*$.
Then the HJB equation
\begin{equation}\label{E1-HJB}
\min_{\xi\in\Act}\;
\bigl[\Lg_\xi \Vst(x) + c(x,\xi)\,\Vst(x)\bigr] \,=\,  \sE_*\,\Vst(x)
\qquad\forall\,x\in\Rd
\end{equation}
has a solution $\Vst\in\Cc^{2}(\RR^{d})$,
satisfying $\inf_{\Rd}\,\Vst>0$, and the following hold:
\begin{enumerate}
\item[\ttup a]
$\sE^x_*=\sE_*=\widehat\lambda_*$ for all $x\in\Rd$.

\item[\ttup b]
Any $v\in\Usm$ that satisfies
\begin{equation}\label{EP1.1A}
\Lg_v \Vst(x) + c\bigl(x,v(x)\bigr)\,\Vst(x)\,=\, 
\min_{\xi\in\Act}\; \bigl[\Lg_\xi \Vst(x) + c(x,\xi)\,\Vst(x)\bigr]
\end{equation}
a.e.\ $x\in\Rd$, is stable, and is optimal, that is, $\sE^v_x=\sE_*$ for all $x\in\Rd$.

\item[\ttup c]
It holds that
\begin{equation*}
\Vst(x) \,=\, \Exp^x_v\Bigl[\E^{\int_{0}^{T}
[c(X_{t},v(X_{t}))-\sE_*]\,\D{t}}\,\Vst(X_T)\Bigr]
\qquad\forall\, (T,x)\in\RR_+\times\Rd\,,
\end{equation*}
for any $v\in\Usm$ that satisfies \cref{EP1.1A}.

\item[\ttup d]
If $\widehat\lambda_*$ is strictly monotone at $c$ on the right,
then there exists a unique positive
solution to \cref{E1-HJB}, up to a multiplicative constant,
and any optimal $v\in\Usm$ satisfies \cref{EP1.1A}.
\end{enumerate}
\end{proposition}

\begin{proof}
A modification of \cite[Lemma~3.2]{AB-18}
(e.g., applying It\^{o}'s formula to the function $f(x)= \abs{x}^{2+2\theta}$)
shows that \cref{EA5.1A} implies that
\begin{equation*}
\limsup_{t\to\infty}\; \frac{1}{t}\;\Exp^x_\xi\bigl[\abs{X_{t}}^{1+\theta}\bigr]
\, =\, 0 \qquad\forall\,\xi\in\Uadm\,.
\end{equation*}
From this point on, the proof follows as in \cite{AB-18}, using \cref{L4.1}.
Indeed, parts (a) and (b) follow from \cite[Theorem~3.4]{AB-18}
by using the above estimate and \cref{L4.1}. 
Since $\inf_{\Rd}\,\Vst>0$, any minimizing selector is recurrent. Moreover,
the twisted diffusion corresponding to the minimizing selector is regular.
Thus part (c) follows from \cite[Theorem~1.5]{AB-18}.
In addition, the hypothesis in (d) implies that for any minimizing selector $v$,
$\lambda_v=\hat\lambda_*$ is right monotone at $c$ which, in turn,
implies the simplicity of the principal eigenvalue by
\cite[Theorem~1.2]{AB-18}. This also implies the last claim
by \cite[Lemma~3.6]{AB-18}.
\end{proof}

\section*{Acknowledgements}
The work of Ari Arapostathis was supported in part by
the National Science Foundation through grant DMS-1715210, in part
the Army Research Office through grant W911NF-17-1-001,
and in part by the Office of Naval Research through grant N00014-16-1-2956
which was approved for public release under DCN \#43-5025-19.
The research of Anup Biswas was supported in part by an INSPIRE faculty fellowship
and DST-SERB grant EMR/2016/004810,
while the work of Vivek Borkar was supported by a J.\ C.\ Bose Fellowship.


\end{document}